\documentclass{article}[11pt]

\usepackage[a4paper,top=4.0cm,bottom=2.0cm,left=2.0cm,right=2.0cm]{geometry}

\usepackage{amsmath,amsfonts,mathrsfs,amsfonts,amsthm}
\usepackage{epsfig,graphicx}
\usepackage{color}

\theoremstyle{plain}
\newtheorem{theorem}{Theorem}[section]
\newtheorem{proposition}[theorem]{Proposition}

\newtheorem{remark}[theorem]{Remark}

\newtheorem{lemma}[theorem]{Lemma}

\begin{document}

\title{Long-time stability and accuracy of the ensemble Kalman-Bucy filter for fully observed processes and small measurement noise}

\author{Jana de Wiljes\thanks{Universit\"at Potsdam, 
Institut f\"ur Mathematik, Karl-Liebknecht-Str. 24/25, D-14476 Potsdam, Germany} \and 
Sebastian Reich\thanks{Universit\"at Potsdam, Institut f\"ur Mathematik, Karl-Liebknecht-Str. 24/25, D-14476 Potsdam, Germany ({\tt sreich@math.uni-potsdam.de}) and University of Reading, Department of Mathematics and Statistics, Whiteknights, PO Box 220, Reading RG6 6AX, UK} \and
Wilhelm Stannat\thanks{TU Berlin, Institut f\"ur Mathematik, Str. des 17. Juni 136, D-10623 Berlin, Germany}
}

\maketitle

\begin{abstract} The ensemble Kalman filter has become a popular data assimilation technique in the geosciences. However, little is known theoretically about its long
term  stability and accuracy. In this paper, we investigate the behavior of an ensemble Kalman-Bucy filter applied to continuous-time filtering problems.
We derive mean field limiting equations as the ensemble size goes to infinity as well as uniform-in-time accuracy and stability results for finite ensemble sizes. The later results require that
the process is fully observed and that the measurement noise is small. We also demonstrate that our ensemble Kalman-Bucy filter is consistent with the classic 
Kalman-Bucy filter for linear systems and Gaussian processes. We finally verify our theoretical findings for the Lorenz-63 system. 
\end{abstract}

\noindent {\bf AMS(MOS) subject classifications.} 65C05, 62M20, 93E11, 62F15, 86A22

%
\section{Introduction}
%

In this paper, we consider the continuous-time filtering problem \cite{sr:jazwinski,sr:crisan} for diffusion processes of type
\begin{equation} \label{sde1}
{\rm d} X_t = f(X_t){\rm d}t + \sqrt{2} C {\rm d}W_t
\end{equation}
and observations, $Y_t$, given by
\begin{equation} \label{obs1}
{\rm d}Y_t = h(X_t){\rm d}t +  R^{1/2} {\rm d} B_t.
\end{equation}
Here $X_t$ denotes the state variable of the $N_x$-dimensional diffusion process with Lipschitz-continuous drift $f:\mathbb{R}^{N_x} \to \mathbb{R}^{N_x}$ and
constant diffusion tensor $D = C C^{\rm T}$ and $C \in \mathbb{R}^{N_x \times N_w}$. 
The observations $Y_t$ are $N_y$-dimensional with forward map $h:\mathbb{R}^{N_x} \to
\mathbb{R}^{N_y}$ and measurement error covariance matrix $R\in \mathbb{R}^{N_y \times N_y}$. Finally, 
$W_t \in \mathbb{R}^{N_w}$ and $B_t \in \mathbb{R}^{N_y}$ denote independent Brownian motion of dimension $N_w$ and $N_y$, respectively.  
It is well-known that the filtering distribution $\pi_t$, i.e., the conditional distribution in $X_t$ for given observations $Y_s$, $s \in [0,t]$, satisfies the
Kushner-Zakai equation \cite{sr:jazwinski,sr:crisan}, which we write as an evolution equation in the expectation values 
\begin{equation} \label{EV}
\pi_t[g] = \int_{\mathbb{R}^{N_x}} g(x) \pi_t(x){\rm d} x
\end{equation}
of smooth and bounded functions $g:\mathbb{R}^{N_x} \to \mathbb{R}$, i.e.
\begin{equation} \label{KS}
{\rm d} \pi_t[g] = \pi_t [ f \cdot \nabla g]{\rm d}t +  \pi_t [ \nabla \cdot D\nabla g]{\rm d}t +  \left(\pi_t[gh] - \pi_t[g]\pi_t[h]\right)^{\rm T} R^{-1}
\left( {\rm d}Y_t - \pi_t[h]{\rm d}t\right).
\end{equation}
In order to have a properly formulated filtering problem, we also have to specify the distribution at initial time $t=0$. 

Popular numerical methods for approximating solutions to (\ref{KS}) include direct finite-difference or finite-element discretizations of
(\ref{KS}) and sequential Monte Carlo methods, also called particle filters \cite{sr:crisan,sr:Doucet}. These methods lead to consistent approximations but are typically 
restricted to low-dimensional problems. In recent years, particle filter methods have become popular, which are applicable to
higher-dimensional problems but are no longer consistent. These include the ensemble Kalman filter (EnKF) \cite{sr:evensen,sr:stuart15,sr:reichcotter15}, 
which is now widely used in the geosciences.

Abstractly spoken, particle filters are defined as follows. First one defines $M$ weighted random variables $X_t^i$, called particles, 
which are i.i.d. at initial time $t=0$ with distribution $\pi_0$, and weights $w_t^i\ge 0$ with $w_0^i = 1/M$ at initial time. A particle filter is then characterized by appropriate evolution laws for the  
particles and the weights. Most known particle filters lead to particles which remain identically distributed while no longer being independent, so called interacting 
particle systems \cite{sr:DelMoral}. If the weights are furthermore kept uniform either through resampling or appropriately defined evolution equations, 
then expectation can be taken with respect to the law $\pi_t^M$ of the $M$th particle and
consistency of a particle filter means that $\lim_{M\to \infty} \pi_t^M[g] \to \pi_t[g]$.

The classic bootstrap filter \cite{sr:arul02} uses (\ref{sde1}) for the evolution of the particles and (\ref{obs1}) for the update of the weights in combination with
an appropriate resampling strategy in order to avoid the weights to degenerate. The EnKF, on the contrary, introduces 
modified evolution equations for the particles which include the observations and keep the weights uniform instead. Most available EnKF 
formulations are stated for discrete-in-time observations \cite{sr:evensen}. While the robust behavior of EnKFs has been demonstrated 
for many applications primarily arising from the geosciences, our theoretical understanding of their long-time stability and accuracy is still rather limited. 
Large sample size limits have been, for example, investigated in \cite{sr:legland11,sr:KM15} and it has been demonstrated that the EnKF converges to the
classic Kalman filter for linear systems (\ref{sde1}), linear observations (\ref{obs1}) and Gaussian initial conditions. 
Using concepts from shadowing, \cite{sr:hunt13} showed that the EnKF is stable and accurate uniformly in time for hyperbolic dynamical systems 
provided the ensemble size is larger than the dimension of the chaotic attractor. Stability and ergodicity of EnKFs have also been studied
in \cite{sr:majda15}. The authors demonstrate that the extended system consisting of (\ref{sde1}), (\ref{obs1}), and the filter algorithm possesses a unique ergodic invariant measure 
provided the existence of an appropriate Lyapunov function can be guaranteed. 
While such ergodicity results of \cite{sr:majda} are important, they do not imply accuracy of a filter. In fact, it is well known, that ensemble Kalman filter can diverge and techniques, 
such as ensemble inflation \cite{sr:evensen}, have been developed in order to stabilize a filter. 
Furthermore, it has been rigorously demonstrated, for example, in \cite{sr:KellyEtAl14} that a judicious 
choice of inflation can lead to uniform-in-time accurate state estimates. At the same time, \cite{KMT15} provides an example of catastrophic filter divergence, i.e.~an exponential 
blow-up of the ensemble systems, for a linear forward map $h(x) = Hx$ with
strongly non-normal operator $H$. 

In this paper, we investigate a time-continuous EnKF formulation which is consistent with the classic Kalman filter in the linear case and which is also stable and accurate uniformly
in time without additional ensemble inflation. In this first study, we will assume for simplicity that the system is fully observable, i.e.~$h(x)=x$ in (\ref{obs1}), and that the measurement errors are small. These assumptions can be relaxed under appropriate assumptions on the stochastic process (\ref{sde1}) and the observation process (\ref{obs1}), well known from the theory of classic Kalman filter theory (i.e.~observability and controlability) \cite{sr:jazwinski}. We will also investigate 
in future work whether the proposed filter formulations can prevent catastrophic filter 
divergence for strongly nonlinear and partially observed systems. 

The specific ensemble Kalman-Bucy filter (EnKBF) formulation, which we will investigate in this paper, is given by drawing $M$ independent realizations (called particles or ensemble members) 
$X_0^i \sim \pi_0$, which then follow the system of differential equations
\begin{equation} \label{EKB1b}
{\rm d}X_t^i = f(X_t^i){\rm d}t  + D (P_t^M)^{-1} (X_t^i - \bar x_t^M){\rm d}t - \frac{1}{2} Q_t^M R^{-1} \left( h(X_t^i){\rm d}t + \bar h_t^M
{\rm d}t - 2{\rm d}Y_t \right)
\end{equation}
for $t\ge 0$. These equations of motion for the particles are closed through the empirical estimates
\begin{equation} \label{EKB2b}
\bar x_t^M = \frac{1}{M}\sum_{i=1}^M X_t^i, \qquad P_t^M = \frac{1}{M-1} \sum_{i=1}^M (X_t^i - \bar x_t^M)
(X_t^i-\bar x_t^M)^{\rm T} ,
\end{equation}
and
\begin{equation} \label{EKB3b}
\bar h_t^M = \frac{1}{M} \sum_{i=1}^M h(X_t^i), \qquad
Q_t^M = \frac{1}{M-1} \sum_{i=1}^M (X_t^i - \bar x_t^M)(h(X_t^i) -\bar h_t^M)^{\rm T}.
\end{equation}
Finally, given a solution of (\ref{EKB1b}), we define the empirical expectation values of a function $g$ and the empirical distribution $\hat \pi_t^M$ 
by
\begin{equation} \label{EmpiricalDistribution}
\bar g^M_t := \hat \pi_t^M[g], \qquad \hat \pi_t^M(x) := \frac{1}{M} \sum_{i=1}^M \delta(x-X_t^i)\,,
\end{equation}
respectively. Here $\delta (\cdot)$ denotes the standard Dirac delta function.
The formulation (\ref{EKB1b}) has been stated first in \cite{sr:br10,sr:br11}. Alternative ensemble Kalman-Bucy formulations include stochastically perturbed
formulations \cite{sr:reich10,sr:stuart15,sr:reichcotter15} and the extended ensemble Kalman-Bucy filter, whose exponential stability and propagation of chaos properties 
have been studied in \cite{sr:DelMoral16}. 

In case $P_t^M$ is not invertible, which is surely the case for $M\le N_x$, the inverse of $P_t^M$ is replaced by its generalized inverse $(P_t^M)^+$. 
This generalization is unproblematic from a mathematical perspective 
since $(P_t^M)^{+}$ gets multiplied by a vector which is in the range of $P_t^M$ and we
show that the equations are well-posed in Section \ref{sec-problem-well-posedness}.
At the same time it is known that $M\ll N_x$ often requires application of 
localization \cite{sr:evensen,sr:reichcotter15} in order to obtain a full rank approximation of the
covariance matrix and to prevent filter divergence. The impact of localization 
has been studied in \cite{sr:tong17} from a rigorous mathematical perspective
for high-dimensional linear systems.

Given the evolution equations (\ref{EKB1b}), one can derive associated evolution equations for the ensemble mean, $\bar x_t^M$, and the ensemble covariance matrix, $P_t^M$.
These are given by
\begin{equation} \label{EKBm}
{\rm d}\bar x_t^M = \bar f_t^M {\rm d}t -  Q_t^M R^{-1} (\bar h_t^M{\rm d}t - {\rm d}Y_t)
\end{equation}
with $\bar f_t^M = \hat \pi_t^M[f]$ and
\begin{equation} \label{EKBP}
\begin{aligned}
\frac{{\rm d}}{{\rm d}t} P_t^M &= \frac{1}{M-1} \sum_{i=1}^M \left\{ 
(f(X_t^i)- \bar f_t)(X_t^i - \bar x_t)^{\rm T} + (X_t^i -\bar x_t)(f(X_t^i)-\bar f_t)^{\rm T} \right\} \,\,+ \\
& \qquad \quad + \,\, \left\{ D(P_t^M)^+ P_t^M + P_t^M (P_t^M)^+ D\right\}  - 
Q_t^M R^{-1} (Q_t^M)^{\rm T}.
\end{aligned}
\end{equation}

We will study the behavior of the EnKBF for fully observed processes, i.e.~$h(x)=x$ and regular 
measurement error covariance matrix $R$ in Sections \ref{sec-problem-well-posedness} and 
\ref{sec-accuracy}. More specifically, it is shown in Section \ref{sec-problem-well-posedness} 
that strong solutions of (\ref{EKB1b}) exist for all times and are unique. This result implies 
that catastrophic filter divergence \cite{KMT15} cannot arise under the setting considered in this 
paper. Next uniform-in-time stability and accuracy of (\ref{EKB1b}) are  proven in Section 
\ref{sec-accuracy} under the additional assumption that $R = \epsilon I$, $\varepsilon > 0$ 
sufficiently small, and that $M > N_x$, i.e., the empirical covariance matrix $P_t^M$ 
is invertible. 
In Sections \ref{sec-linear} and \ref{sec-asymptotic}, we return to the filtering problem for 
general observation operator, $h$, 
and measurement error covariance matrix $R$. It is demonstrated in Section \ref{sec-linear} that
in the case of linear systems, (\ref{EKBm}) and (\ref{EKBP}) are consistent with the classic 
Kalman-Bucy filtering equations \cite{sr:jazwinski}. 

Note that this does not imply that the empirical distribution of the extended ensemble Kalman-Bucy 
filter is asymptotically normal. In fact, we will identify in Section \ref{sec-asymptotic} its 
asymptotic distribution for $M\to\infty$. To this end we will prove in Theorem \ref{MFConvergence} 
that the ensemble $X_t^i$, $1\le i\le M$, converges as $M\to\infty$ to independent solutions  
$\hat{X}_t^i$, $i=1, 2, 3, \ldots$, of the following McKean-Vlasov equation  
\begin{equation} \label{EKB1a}
{\rm d} \hat X_t = f(\hat X_t){\rm d}t + D ({\cal P}_t)^{-1} (\hat X_t - \bar x_t){\rm d}t - 
\frac{1}{2} {\cal Q}_t R^{-1} \left( h(\hat X_t){\rm d}t + \bar h_t {\rm d}t  
- 2{\rm d}Y_t \right),
\end{equation}
with $\bar x_t = \hat \pi_t [x]$, $h_t = \hat \pi_t [h]$,
\begin{equation} \label{EKB2a}
{\cal P}_t = \mbox{Cov}\,(\hat X_t,\hat X_t), \qquad {\cal Q}_t  
= \mbox{Cov}  \,(\hat X_t,h(\hat X_t)) \, .
\end{equation}
Here $\hat \pi_t$ denotes the distribution of $\hat X_t$.  

Using It\^o's formula, it is then easy to derive from \eqref{EKB1a} the weak formulation of the 
nonlinear Fokker-Planck equation driving the distribution $\hat \pi_t$ of $\hat{X}_t$  
\begin{equation} 
\label{limit1}
\begin{aligned}
{\rm d} \hat  \pi_t [g] &= \hat \pi_t \left[ \nabla g \cdot \left\{ f{\rm d}t + D {\cal P}^{-1}_t 
(x- \hat \pi_t [x]){\rm d}t - \frac{1}{2} {\cal Q}_t R^{-1} (h(x){\rm d}t +
\hat \pi_t [h]{\rm d}t - 2{\rm d}Y_t ) \right\}\right] \\
&  \qquad + \,\hat \pi_t \left[ \frac{1}{2} \nabla\cdot {\cal Q}_t R^{-1} 
{\cal Q}_t^{\rm T} \nabla g \,{\rm d} t \right] \, .
\end{aligned}
\end{equation}
Note the difference between \eqref{limit1} and the Kushner-Zakai equation \eqref{KS}. 

Some numerical results, supporting our theoretical estimates, will be
presented in Section \ref{sec-numerics} using a stochastically perturbed Lorenz-63 system \cite{sr:lorenz63,sr:stuart15}.

%
\section{Well-posedness of the ensemble Kalman-Bucy filter for fully observed processes} \label{sec-problem-well-posedness}
%

In this section, we specify the problem setting which is investigated in detail in this paper. We will also derive a first well-posedness result for the system (\ref{EKB1b})--(\ref{EKB3b}) implying that
the filter is not subject to catastrophic filter divergence. More specifically, we assume that the process is fully observed, i.e.~$h(x) = x$, that the diffusion tensor $D$ has full rank, and that the drift function $f$ is globally Lipschitz continuous. Since the ensemble size, $M$, will be fixed in this section, we also drop the superscript $M$ in (\ref{EKB1b}). Hence (\ref{EKB1b}) is replaced by
\begin{equation} \label{EKB1c}
{\rm d}X_t^i = f(X_t^i){\rm d}t  + D P_t^{+} (X_t^i - \bar x_t){\rm d}t - \frac{1}{2} P_t R^{-1} 
\left( X_t^i{\rm d}t + \bar x_t {\rm d}t - 2{\rm d}Y_t \right)\,,
\end{equation}
$i=1,\ldots,M$. The standard inner product in $R^{N_x}$ will be denoted 
by $\langle \cdot,\cdot \rangle$ and we recall that
\begin{equation}
\langle a,b\rangle = \mbox{tr}\,(b a^{\rm T}).
\end{equation}
Hence we quickly verify that
\begin{equation}
\frac{1}{M-1} \sum_{i=1}^M \langle X_t^i-\bar x_t,DP_t^{+} (X_t^i-\bar x_t)\rangle = \mbox{tr} \,(DP_t^{+} P_t) 
\end{equation}
and
\begin{equation}
\frac{1}{M-1} \sum_{i=1}^M \langle X_t^i-\bar x_t,P_t R^{-1}(X_t^i-\bar x_t)\rangle  
= \mbox{tr} \,(P_t R^{-1} P_t) = \|R^{-\frac 12} P_t\|^2_{\rm F} .
\end{equation} 
Here $\|A\|_{\rm F}$ denotes the Frobenius norm of a matrix $A$. We also introduce the notation $\langle A,B\rangle = \mbox{tr}\,
(B A^{\rm T})$, i.e.~$\|A\|_{\rm F}^2 = \langle A,A\rangle$.

We now investigate the $l_2$-norm of the ensemble deviations from the mean, i.e.
\begin{equation} \label{l2norm}
V_t = \frac{1}{M-1} \sum_{i=1}^M \|X^i_t - \bar x_t\|^2,
\end{equation}
which satisfies the evolution equation
\begin{equation} \label{l2-norm1}
\begin{aligned}
\frac{1}{2} \frac{{\rm d}V_t}{{\rm d}t} &= \frac{1}{M-1} \sum_{i=1}^M \langle X_t^i-\bar x_t,f(X_t^i)-\bar f_t\rangle +
\frac{1}{M-1} \sum_{i=1}^M \langle X_t^i-\bar x_t,DP_t^{+} (X_t^i-\bar x_t)\rangle 
\\ & \qquad \quad -\,\,
\frac{1}{2\varepsilon} \frac{1}{M-1} \sum_{i=1}^M \langle X_t^i-\bar x_t,P_t R^{-1} 
 (X_t^i-\bar x_t)\rangle \\
& = \frac{1}{M-1} \sum_{i=1}^M \langle X_t^i-\bar x_t,f(X_t^i) -  \bar f_t \rangle  + \mbox{tr}\,(DP_t^{+} P_t) - 
\frac{1}{2} \|R^{-\frac 12} P_t\|^2_{\rm F} \\
& =  \frac{1}{M-1} \sum_{i=1}^M \langle X_t^i-\bar x_t,f(X_t^i)-  f(\bar x_t)  \rangle  + \mbox{tr}\,(DP_t^{+} P_t) - 
\frac{1}{2} \|R^{-\frac 12} P_t\|^2_{\rm F}\, . 
\end{aligned}
\end{equation}
Here we have used 
\begin{equation}
\sum_i \langle X_t^i - \bar x_t , f(\bar x_t) - \bar f_t \rangle = 0
\end{equation} 
and that the evolution equation (\ref{EKBm}) for the mean, $\bar x_t$, reduces to
\begin{equation}\label{EKBm2}
{\rm d}\bar x_t = \bar f_t {\rm d}t -  P_t R^{-1}(\bar x_t {\rm d}t - {\rm d}Y_t)
\end{equation}
in our setting.

\begin{lemma}
\label{ControlFrobeniusNorm}
The Frobenius norm of $P_t$ satisfies
\begin{equation}
\frac{1}{\sqrt{M}} V_t \le \|P_t\|_{\rm F} \le V_t\,.
\end{equation}
\end{lemma}

\begin{proof} 
We first note the following identity: 
\begin{equation}
\|P_t\|_{\rm F}^2 
 = \frac 1{(M-1)^2} \sum_{i,j}  
\langle X_t^i - \bar x_t , X_t^j - \bar x_t \rangle^2 \, . 
\end{equation} 
For the proof of the upper bound it is now sufficient to observe that 
\begin{equation}
\begin{aligned}
\frac 1{(M-1)^2} \sum_{i,j}  
\langle X_t^i - \bar x_t , X_t^j - \bar x_t \rangle^2
& \le \frac 1{(M-1)^2} \sum_{i,j}  
\|X_t^i - \bar x_t\|^2 \| X_t^j - \bar x_t\|^2\\
&
= \left( \frac 1{M-1} \sum_i \|X_t^i - \bar x_t\|^2\right)^2\, .
\end{aligned}
\end{equation}
For the proof of the lower bound observe that 
\begin{equation}
\frac 1{(M-1)^2} \sum_{i,j}  
\langle X_t^i - \bar x_t , X_t^j - \bar x_t \rangle^2 
\ge \frac 1{(M-1)^2} \sum_{i}  \|X_t^i - \bar x_t\|^4 
\ge \frac 1M \left( \frac 1{M-1} \sum_i \|X_t^i - \bar x_t\|^2\right)^2\, .
\end{equation}
\end{proof}

\begin{remark} We recall the  standard relations between the Frobenius and the spectral norm of a matrix $A$, i.e.,
\begin{equation}
\|A\| \le \|A\|_{\rm F} 
\end{equation}
and
\begin{equation}
 \|A\|_{\rm F} \le \sqrt{N_x} \|A\|\,.
\end{equation} 
\end{remark}

We are now ready to obtain uniform-in-time upper and lower bounds on $V_t$. First, we can estimate the first term of \eqref{l2-norm1} from above and from below as follows: 
\begin{equation}
\begin{aligned} 
 \frac 1{M-1} \sum_i \langle X_t^i - \bar x_t , f(X_t^i) - f(\bar x_t) \rangle 
& \le L_+ \frac 1{M-1} \sum_i \| X_t^i - \bar x_t \|^2 = L_+ V_t  
\end{aligned} 
\end{equation}
and
\begin{equation}
\begin{aligned} 
\frac 1{M-1} \sum_i \langle X_t^i - \bar x_t , f(X_t^i) - f(\bar x_t) \rangle 
 \ge L_- \frac 1{M-1} \sum_i \| X_t^i - \bar x_t \|^2 = L_- V_t  ,
\end{aligned} 
\end{equation}
respectively, where 
\begin{equation}
\begin{aligned} 
L_+  := \sup_{x\neq y} \frac{\langle f(x) - f(y) , x-y\rangle}{\|x-y\|^2} \qquad 
L_-  := \inf_{x\neq y} \frac{\langle f(x) - f(y) , x-y\rangle}{\|x-y\|^2} 
\end{aligned} 
\end{equation}
are the upper and lower control on the \lq\lq dissipativity\rq\rq{} constant of $f$. 
We clearly have $L_+ \le \|f\|_{\rm Lip}$ and $L_-\ge - \|f\|_{\rm Lip}$ for globally Lipschitz continuous $f$ with
Lipschitz constant $\|f\|_{\rm Lip}$. Provided $V_t\not=0$, we also find that
\begin{equation}
\lambda^{\rm min}(D)  \le \mbox{tr}\,(DP_t^{+}P_t) \le {\rm tr} \,(D)\,.
\end{equation}
Here, $\lambda^{\min}(A)$ and $\lambda^{\max}(A)$ denote the smallest and 
largest singular values of a matrix $A$, respectively. 

Finally, the third term in \eqref{l2-norm1} can be estimated from above and from below using  
\begin{equation} 
\label{ControlFrobeniusNormGeneralLowerBound} 
\lambda^{\min} (R^{-1})\|P_t\|^2_{\rm F} = \mbox{tr} (P_t \lambda^{\min} (R^{-1}) P_t) 
\le \mbox{tr} (P_t R^{-1} P_t) = \|R^{-\frac 12} P_t\|_{\rm F}^2 
\end{equation} 
and 
\begin{equation} 
\label{ControlFrobeniusNormGeneralUpperBound} 
\|R^{-\frac 12} P_t\|_{\rm F}^2  = \mbox{tr} (P_t R^{-1} P_t) 
\le  \mbox{tr} (P_t \lambda^{\max} (R^{-1}) P_t) 
= \lambda^{\max} (R^{-1}) \,\mbox{tr} (P_t ^2) 
= (\lambda^{\min} (R))^{-1} \|P_t\|^2_{\rm F}\, , 
\end{equation} 
which follow from the inequalities $P_t (\lambda^{\min} (R^{-1}) P_t\le P_t R^{-1}P_t \le 
P_t \lambda^{\max} (R^{-1} ) P_t$, where $\le$ is meant in the sense of (symmetric) positive 
(semi-) definite matrices. 

Inserting these estimates and the previous two identities into \eqref{l2-norm1} we first obtain the upper bound  
\begin{equation} 
\label{ControlDifferentialL2NormUpperBoundGeneral} 
\frac 12 \frac{{\rm d}V_t}{{\rm d}t} \le L_+ V_t + {\rm tr} \,(D)
- \frac{\lambda^{\min} (R^{-1})}{2M} V_t^2 \, . 
\end{equation} 
This implies that $V_t \le \max\{V_0, \lambda^{\max} (R) L_+ M  
+ \sqrt{\left(  \lambda^{\max} (R) ML_+ \right)^2 + 2\lambda^{\max} (R) M \,{\rm tr} \,(D)}\}$ 
uniformly in $t$. Similarly, we obtain the lower bound  
\begin{equation} 
\label{ControlDifferentialL2NormLowerBoundGeneral} 
\frac 12 \frac{{\rm d}V_t}{{\rm d}t} \ge L_- V_t + \lambda^{\rm min}(D) 
- \frac{\lambda^{\max} (R^{-1})}{2} V_t^2 \, ,
\end{equation} 
which implies that $V_t \ge \min\{V_0, \lambda^{\min} (R) L_-  
+ \sqrt{\left( \lambda^{\min} (R)L_-\right)^2 + 2\lambda^{\min} (R) \lambda^{\rm min}(D)}\}$ 
uniformly in $t$ and $V_t >0$ provided $V_0>0$.

\begin{theorem} 
Assume that the drift term $f$ in (\ref{sde1}) is globally Lipschitz continuous and satisfies a linear growth condition
\begin{equation}
\|f(x)\| \le \tilde c_1 (1+\|x\|)
\end{equation}
for an appropriate $\tilde c_1>0$. 
Then the system (\ref{EKB1c}) together with (\ref{EKB2b})-(\ref{EKB3b}) possesses strong solutions for all times $t\ge 0$. 
\end{theorem}

\begin{proof}
We can decompose the equations (\ref{EKB1c}) into ordinary differential equations in $X_t^i-\bar x_t$, $i=1,\ldots,M$ and Equation
(\ref{EKBm2}) for the mean, $\bar x_t$. Since the $l_2$-norm, $V_t$, remains bounded, the equations in $X_t^i - \bar x_t$ are well-posed. 
Furthermore, since $\|P_t\|$ remains bounded as well, the combined drift term in (\ref{EKBm2}), written as
\begin{equation}
{\rm d}\bar x_t = f(\bar x_t) {\rm d}t  + b_t {\rm d}t -   P_t R^{-1} (\bar x_t {\rm d}t - {\rm d}Y_t),
\end{equation}
with $b_t = \bar f_t - f(\bar x_t)$, is Lipschitz continuous in $\bar x_t$ and, hence, satisfies a 
linear growth condition, i.e.
\begin{equation}
\begin{aligned}
\| f(\bar x_t) + b_t - P_t R^{-1}\bar x_t\| \le \| f(\bar x_t) - P_t R^{-1}\bar x_t\| + \|\bar f_t - f(\bar x_t)\| \le \tilde c_2  ( 1 + \|\bar x_t\| )
\end{aligned}
\end{equation}
for an appropriate $\tilde c_2 >0$, and, hence, strong solutions to (\ref{EKBm2}) exist for all times \cite{sr:Oksendal}.
\end{proof}

\begin{remark} \label{UnifMBds}
For the analysis of the asymptotic behavior of $M\to\infty$ the upper bound on $V_t$ is not sufficient, because it diverges as $M\to\infty$. However, since 
we need a control only locally in time, we can use \eqref{ControlDifferentialL2NormUpperBoundGeneral} 
to estimate $\frac 12 \frac{d}{dt} V_t \le L_+ V_t + {\rm tr} \,(D)$ 
which implies the upper bound 
\begin{equation} 
\label{UniformMUpperBound} 
V_t\le e^{2L_+ t} \left( V_0 + \frac{{\rm tr} \,(D)}{L_+} \right)\, , 
\end{equation}  
which becomes uniform in $M$ (but of course not in $t$) if the particles at time $t=0$ are chosen with uniform upper bound on $V_0 = V^M_0$. 
\end{remark}

%
\section{Accuracy of the ensemble Kalman-Bucy filter for finite ensemble sizes and small measurement noise} \label{sec-accuracy}
%

The goal of this section is to derive bounds on the estimation error
\begin{equation} \label{error-e}
e_t = X_t^{\rm ref} - \bar x_t,
\end{equation}
where $X_t^{\rm ref}$ denotes the reference trajectory of (\ref{sde1}) which generated the data. 
We again restrict the discussion to fully observed processes and globally Lipschitz-continuous drift
functions $f$. In addition, we assume the error covariance to be of the type $R = \varepsilon I$ 
with sufficiently small $\varepsilon > 0$, implying
\begin{equation}
{\rm d}Y_t = X_t^{\rm ref} {\rm d}t + \sqrt{\varepsilon} {\rm d}B_t\,,
\end{equation}
and that $P_t^M$ is invertible which necessitates that $M>N_x$. We 
drop the superscript $M$ from all relevant quantities throughout this section, 
as we are interested in the accuracy of the filter for fixed ensemble size, $M$. 

We find that the estimation error satisfies the evolution equation
\begin{equation} \label{error-de}
{\rm d} e_t = f(X_t^{\rm ref}) {\rm d} t + \sqrt{2} C {\rm d} W_t - \bar f_t {\rm d}t 
- \frac{1}{\varepsilon} P_t (e_t {\rm d} t + \varepsilon^{1/2} {\rm d} B_t).
\end{equation}
We introduce the squared estimation error norm $E_t = \|e_t\|^2/2 = \langle e_t,e_t \rangle/2$. 
Then Ito's formula implies that 
\begin{align} \nonumber
{\rm d} E_t  
& = \langle f(X_t^{\rm ref}) - \bar{f}_t, X_t^{\rm ref}- \bar x_t \rangle\, {\rm d} t - \frac 1\varepsilon 
\langle e_t , P_t e_t\rangle\, {\rm d} t \\ 
& \qquad + \langle e_t , \sqrt{2} C {\rm d} W_t\rangle -
\frac 1{\sqrt{\varepsilon}} \langle e_t,P_t {\rm d} B_t\rangle + \mbox{tr}\,(D)\,{\rm d} t + 
\frac 1{2\varepsilon} \mbox{ tr} (P_t^2)\, {\rm d} t \, , 
\end{align} 
which can be rewritten as
\begin{equation} \label{error1}
{\rm d} E_t = {\cal E}_t {\rm d} t + {\rm d}M_t 
\end{equation}
with
\begin{align}
{\cal E}_t &= \langle X_t^{\rm ref} - \bar x_t,f(X_t^{\rm ref})-\bar f_t \rangle - 
\frac{1}{\varepsilon} \langle e_t ,P_t e_t \rangle + \mbox{tr}\,(D) + \frac{1}{2\varepsilon}\|P_t\|_{\rm F}^2
\end{align}
and the martingale 
$$ 
M_t = \int_0^t \langle e_s , -\varepsilon^{-1/2} P_s{\rm d}B_s + \sqrt{2} C{\rm d} W_s\rangle \, ,\quad 
t\ge 0\, . 
$$ 

To make further progress we need bounds for the smallest and largest singular values 
$\lambda_t^{\rm min} = \lambda^{\min} (P_t)$ and $\lambda_t^{\rm max} = \lambda^{\max} (P_t)$ of 
$P_t$, respectively. An upper bound for the largest singular value has already been derived in 
Section \ref{sec-problem-well-posedness}, since $\lambda_t^{\rm max} = \|P_t\| \le  
\|P_t \|_{\rm F}\le V_t$. Since $P_t$ is assumed to be invertible, the explicit evolution equation 
for $P_t$ reduces to 
\begin{align} 
\frac d{dt}P_t  
& = \frac 1{M-1} \sum_i (f(X_t^i) - \bar{f} (t)) (X_t^i - \bar{x}_t)^T 
  +  (X_t^i - \bar{x}_t) (f(X_t^i) - \bar{f}(t))^T + 2D - \frac 1{\varepsilon} P_t^2    \,.
\end{align} 
Next we make use of the fact that $P_t$ can be diagonalized, i.e., there are orthogonal matrices 
$Q_t$ and diagonal matrices $\Lambda_t$ such that
\begin{equation}
P_t = Q_t^{\rm T} \Lambda_t Q_t.
\end{equation}
While the orthogonal matrices $Q_t$ are in general only continuous in $t$, the diagonal matrix of 
singular values can be chosen to be differentiable in $t$ \cite{sr:Rellich}. 
As shown in \cite{sr:DE06}, the evolution equation for diagonal matrix 
of eigenvalues, $\Lambda_t$, is of the form
\begin{align} \label{eq2}
\frac{{\rm d}\Lambda_t}{{\rm d}t} &= \mbox{diag}\,(Q_t U_t Q_t^{\rm T}) +  2 \mbox{diag}\,(Q_t D Q_t^{\rm T}) - \frac{1}{\varepsilon} \Lambda_t^2 
\end{align}
with
\begin{equation}
U_t :=  \frac{1}{M-1} \sum_i \left\{ f(X_t^i) - f(\bar{x}_t) \right\} 
\left\{ X_t^i- \bar{x}_t \right\}^{\rm T}   +   
\left\{ X_t^i- \bar{x}_t\right\}   \left\{ f(X_t^i) -  f(\bar{x}_t) \right\}^{\rm T}.
\end{equation}
Here ${\rm diag}\,(A)$
denotes a diagonal matrix with diagonal entries equal to the diagonal of $A$. More specifically, the diagonal entries of
$\mbox{diag}\,(Q_t U_t Q_t^{\rm T})$ are given by
\begin{equation}
\left( \mbox{diag}\,(Q_t U_t Q_t^{\rm T}) \right)_{ii} = e_i^{\rm T} Q_t U_t Q_t^{\rm T} e_i
\end{equation}
where $e_i \in \mathbb{R}^{N_x}$ denotes the $i$th basis vector in $\mathbb{R}^{N_x}$.

Next we derive the following estimate using the fact that $f$ is globally Lipschitz continuous. 
Then, given any unit vector $v$, it holds that
\begin{equation}
\begin{aligned}
 | \frac 1{M-1} \sum_i  & \langle f(X_t^i) - f(\bar x_t), v\rangle   
\langle X_t^i - \bar x_t, v\rangle | \\
& \le \left( \frac 1{M-1} \sum_i \langle f(X_t^i) - f(\bar{x}_t), v
\rangle^2\right)^{\frac 12} \left( \frac 1{M-1} \sum_i \langle X_t^i - \bar{x}_t, v
\rangle^2\right)^{\frac 12} \\
& \le \|f\|_{\rm Lip} V_t \le \|f\|_{\rm Lip}\sqrt{N_x M} \|P_t\|  \,,
\end{aligned} 
\end{equation}
where we have used $V_t \le \sqrt{N_x M} \|P_t\|$.

Hence setting $v = Q_t^{\rm T} e_i$, we obtain
\begin{equation}
|\left( \mbox{diag}\,(Q_t U_t Q_t^{\rm T}) \right)_{ii}| \le 2\|f\|_{\rm Lip}\sqrt{N_x M} \|P_t\| 
= 2 \|f\|_{\rm Lip}\sqrt{N_x M} \lambda_t^{\rm max}\,.
\end{equation}
Since $\lambda_t^{\rm max} = (\Lambda_t)_{ii}$ for some index $i$, we hence deduce that
\begin{align} \label{EstimateMax}
\frac{{\rm d}\lambda_t^{\rm max}}{{\rm d}t} &\le 2 \|f\|_{\rm Lip}\sqrt{N_x M} \lambda_t^{\rm max} + 2\lambda^{\max} (D) - \frac{(\lambda_t^{\rm max})^2}{\varepsilon} \, .
\end{align}
This implies that
\begin{equation}
\lambda_t^{\rm max} \le \max \left\{ \lambda_0^{\rm max},\varepsilon  \|f\|_{\rm Lip}\sqrt{N_x M} + \sqrt{ \varepsilon^2 \|f\|_{\rm Lip}^2 N_x M + 2\varepsilon \lambda^{\max} (D)} \right\}\,.
\end{equation}
Hence we have shown the following 

\begin{lemma} \label{lem-2} (upper bound on spectral radius of $P_t$) There is a constant 
\begin{equation}
C_1 = C_1(\|f\|_{\rm Lip},M,N_x,D,\varepsilon_0)
\end{equation}
such that
$\lambda_0^{\rm max} \le C_1 \varepsilon^{1/2}$ at initial time $t=0$ implies $\lambda_t^{\rm max} \le C_1 \varepsilon^{1/2}$ for all
times and all $\varepsilon \le \varepsilon_0$.
\end{lemma}

\noindent
We now use our upper bound on $\lambda_t^{\rm max}  = \|P_t\|_2$ from Lemma \ref{lem-2} in order to get the  estimate
\begin{equation}
|\left( \mbox{diag}\,(Q_t U_t Q_t^{\rm T} \right)_{ii}| \le 2 L \sqrt{N_x M} C_1 \varepsilon^{1/2}.
\end{equation}
Hence, we deduce that
\begin{equation} \label{EstimateMin}
\frac{{\rm d} \lambda_t^{\rm min}}{{\rm d} t} \ge - 2  \|f\|_{\rm Lip}\sqrt{N_x M} C_1 \varepsilon^{1/2} + 2\lambda^{\min} (D) - 
\frac{(\lambda_t^{\rm min})^2} {\varepsilon} 
\end{equation}
and
\begin{equation}
\lambda_t^{\rm min} \ge \min \left\{ \lambda_0^{\rm min}, -\varepsilon^{3/2} \|f\|_{\rm Lip} C_1 \sqrt{N_x M} + \sqrt{ \varepsilon^3 \|f\|_{\rm Lip}^2 C_1^2 N_x M + 2\varepsilon \lambda^{\min} (D)} \right\}\,,
\end{equation}
which implies the desired lower bound on $\lambda_t^{\rm min}$. Here $\lambda^{\min} (D)$ denotes the smallest eigenvalue of $D$. 
We now fix $\varepsilon_0>0$ such that 
\begin{equation}
-\varepsilon_0^{3/2}\|f\|_{\rm Lip}  C_1 \sqrt{N_x M} + \sqrt{ \varepsilon_0^3 \|f\|_{\rm Lip}^2 C_1^2N_x M + \varepsilon_0 \lambda^{\min} (D)} > 0\,.
 \end{equation}

\begin{lemma} \label{lem-3} (lower bound on smallest singular value of $P_t$)
 There is a constant 
 \begin{equation}
 C_2 = C_2(\|f\|_{\rm Lip},M,N_x,D,\varepsilon_0)
 \end{equation} 
 such that $\lambda_0^{\rm min} \ge C_2 \varepsilon^{1/2}$ 
at initial time $t=0$ implies $\lambda_t^{\rm min} \ge C_2 \varepsilon^{1/2}$ for all $t>0$ and all $\varepsilon \le \varepsilon_0$.
\end{lemma}

\begin{remark}
The upper and lower bounds for the largest and smallest, respectively, eigenvalue of $P_t$ depend on the 
ensemble size, $M$. This dependence can be eliminated for the price of the estimates no longer being valid uniformly in
time. We now derive such $M$-independent upper and lower bounds. Let us assume that
\begin{equation}
\|f\|_{\rm Lip} V_s \le \lambda^{\max} (D)
\end{equation}
for all $s \in [0,t]$. Such a bound can be found because of (\ref{UniformMUpperBound}) and for $\varepsilon$ sufficiently small, 
i.e.~$\varepsilon \le \varepsilon_t$. Then (\ref{EstimateMax}) implies that
\begin{equation}
\lambda_s^{\rm max} \le 2 \left(\lambda^{\max} (D) \varepsilon\right)^{1/2}
\end{equation}
for all $s \in [0,t]$ and all $\varepsilon \le \varepsilon_t$. Similarly, (\ref{EstimateMin}) implies that
\begin{equation}
\lambda_s^{\rm min} \ge \left( 2\lambda^{\min} (D) \varepsilon \right)^{1/2}\,.
\end{equation}
Hence we have traded the $M$-dependent constants $C_1$ and $C_2$ in the previous two lemmas by $M$-independent
constants $\tilde C_1 = 2\lambda^{\max} (D)^{1/2}$ and $\tilde C_2 = \lambda^{\min} (D)^{1/2}$, respectively. However, the estimates hold 
for $\varepsilon \le \varepsilon_t$ only, where the upper bound $\varepsilon_t = 
\varepsilon_t(\|f\|_{\rm Lip},D)$ decreases in time. 
\end{remark}
 
\noindent
 The upper and lower bounds of the eigenvalues  of $P_t$ obtained in the previous 
two lemmas hold with constants  $C_1$ and $C_2$ independent of the driving Wiener processes. They 
only depend on the initial conditions (which might be random), but we can impose deterministic 
bounds on the spectral radius of the covariance matrix.  
Hence we can take expectations on both sides of (\ref{error1}) in order to obtain the following 
integral inequality
\begin{equation}
\begin{aligned}
\mathbb{E}\left[ E_t\right]  &\le \mathbb{E}\left[ E_0\right] + \int_0^t  \mathbb{E} \left[ \mathcal E_s \right]\, {\rm d} s  \\
&
\le  \mathbb{E}\left[ E_0\right] + \int_0^t  \mathbb{E} \left[ \mbox{tr}\,(D) + N_x C_1^2  + 2 \varepsilon^{1/4}  \|f\|_{\rm Lip} C_1^{1/2} 
(N_x M)^{1/4} E_s^{1/2} -  2\frac{C_2 - \varepsilon^{1/2} \|f\|_{\rm Lip}}{\varepsilon^{1/2}} E_s\right]  \, {\rm d}s  \, ,
\end{aligned}
\end{equation}
where we used
\begin{equation}
\begin{aligned}
\langle f(X_t^{\rm ref}) - \bar f_t,X_t^{\rm ref} - \bar x_t\rangle 
& = \langle f(X_t^{\rm ref})- f(\bar x_t),X_t^{\rm ref} - \bar x_t\rangle +
\langle f(\bar x_t) - \bar f_t,X_t^{\rm ref} - \bar x_t\rangle \\
& \le 2 L_+ E_t  + \sqrt{2} \|f\|_{\rm Lip} V_t^{1/2} E_t^{1/2} \\ 
& \le 2 \|f\|_{\rm Lip}\left( E_t + \varepsilon^{1/4} C_1^{1/2} (MN_x)^{1/4} E_t^{1/2} \right)\,.
\end{aligned}
\end{equation}
The next step is to close the right hand side in $\mathbb{E}[E_s]$. To this end, we first derive the following $\omega$-wise estimate 
\begin{equation} 
\label{est_omega_wise}
\begin{aligned} 
\mathcal E_s &\le \left( \mbox{tr}\,(D) + N_x C_1^2  + 2\varepsilon^{1/4}  \|f\|_{\rm Lip} C_1^{1/2} (N_x M)^{1/4}  E_s^{1/2} - 
2\frac{C_2 - \varepsilon^{1/2} \|f\|_{\rm Lip}}{\varepsilon^{1/2}} E_s\right)  \\
& \le C_3 + \varepsilon^{1/4}  C_4 E_s^{1/2} - 
2\frac{C_2 - \varepsilon^{1/2} \|f\|_{\rm Lip}}{\varepsilon^{1/2}} E_s  \\
& \le \left( C_3 + \varepsilon \frac{C_4^2}{C_2}\right) 
  - \frac{C_2 - 2\varepsilon^{1/2}\|f\|_{\rm Lip} }{\varepsilon^{1/2}} E_s \\
& =: \Phi \left(  E_s\right)  
\end{aligned} 
\end{equation} 
for $C_3 = \mbox{tr}\,(D) + N_x C_1^2$, $C_4 = 2\|f\|_{\rm Lip}C_1^{1/2} (N_x M)^{1/4}$, and a  linear function $\Phi (E_s)$. Taking expectations and using
$\mathbb{E}\left[\Phi(E_s)\right] = \Phi \left(\mathbb{E}[E_s]\right)$ we arrive at the integral inequality 
\begin{equation}
\mathbb{E}\left[ E_t\right]  \le \mathbb{E}\left[ E_0\right] + \int_0^t  \Phi\left(  \mathbb{E}\left[ E_s\right]\right) \, {\rm d}s 
\end{equation}
and we can now apply the Gronwall lemma or comparison techniques for integral inequalities. More 
precisely, let $\alpha = \varepsilon^{-1/2}(C_2 - 2\varepsilon^{1/2} \|f\|_{\rm Lip})>0$, then the time-dependent Ito's-formula implies 
that 
\begin{equation}
e^{\alpha t} \mathbb{E}\left[ E_t\right] 
\le \mathbb{E}\left[ E_0\right] + \int_0^t e^{\alpha s} \left( C_3 + \varepsilon
\frac{C_4^2}{C_2}\right) \,{\rm d} s  
\end{equation}
and, hence,  
\begin{equation}
\mathbb{E}\left[ E_t\right] 
\le e^{-\alpha t} \mathbb{E}\left[ E_0\right] + \alpha^{-1} K
\end{equation}
with $K:=  C_3 + \varepsilon \frac{C_4^2}{C_2}$. Note that $\alpha^{-1} = 
\mathcal {O} (\varepsilon^{\frac 12} )$. Hence we have shown the following

\begin{theorem}
\label{thm:MSE} 
(estimation error) If the measurement error variance $\varepsilon$ is chosen 
sufficiently small, the initial ensemble is chosen such that $P_0$ is invertible
and the bounds of Lemmas \ref{lem-2} and \ref{lem-3} are satisfied at initial time, then the  
mean squared estimation error  is of order $\varepsilon^{1/2}$ asymptotically in time.  
\end{theorem}

\noindent
Using Markov's inequality the above estimate on the measurement error now yields for fixed 
$t$ the following estimate 
\begin{equation}
\mathbb{P}\left[ E_t \ge c\varepsilon^q\right] \le \frac 1{c\varepsilon^q} \mathbb{E}\left[ E_t\right]
 = \mathcal{O} \left( \varepsilon^{1/2-q}\right) \, . 
\end{equation}
In particular, for any $q\in (0, 1/2)$ the estimation error $E_t = \|e_t\|^2/2$ is of order 
$\mathcal{O}\left( \varepsilon^q \right)$ with probability close to one. Note that this does not 
imply that for a given realization of the EnKBF, the estimation error $E_t$ will be small all the 
time, i.e.~that $\sup_{t\ge 0} E_t$ (or $\max_{t\in [0,T]} E_t$) is of order 
$\mathcal{O}\left( \varepsilon^q \right)$ with probability close to one. This latter statement 
requires a pathwise control, i.e.~a (locally) uniform in time control of $E_t$, which we will 
derive in the next step. To this end note that \eqref{error1} together with the inequality 
\eqref{est_omega_wise} imply the pathwise estimate 
\begin{equation}
\begin{aligned}
E_t & \le e^{-\alpha t}E_0 + \frac{K}{\alpha} (1-e^{-\alpha t}) + \int_0^t e^{-\alpha (t-s)}\, {\rm d}M_s \\
& =  e^{-\alpha t}E_0 + \frac{K}{\alpha} (1-e^{-\alpha t}) + 
e^{-\alpha t} M_t + \alpha\int_0^t e^{-\alpha (t-s)} (M_t - M_s)\, {\rm d}s \, , 
\end{aligned} 
\end{equation}
hence 
\begin{equation}
\label{sup_est-omega_wise} 
\sup_{t\le T} E_t \le \left( E_0 + \frac{K}{\alpha}\right) + \sup_{t\le T} \left( e^{-\alpha t} |M_t| +  \alpha \int_0^t e^{-\alpha (t-s)}|M_t - M_s|\, {\rm d} s \right) \, . 
\end{equation} 
In order to control the third term, first note that the quadratic variation of the martingale is given as 
\begin{equation} 
\langle M\rangle_t = \int_0^t \varepsilon^{-1} \|P_s e_s\|^2 + 2 \|C e_s\|^2 \, {\rm d}r \, , 
\end{equation} 
so that  
\begin{equation} 
\langle M\rangle_t - \langle M\rangle_s = \int_s^t \varepsilon^{-1} \|P_r e_r\|^2 + 2 \|Ce_r\|^2\, 
{\rm d}r \le (C_1 + 2) \int_s^t E_r \, {\rm d}r \, . 
\end{equation} 
In the following let $L_{T,\delta} := \sup_{0\le s < t\le T} |M_t - M_s|/\left( \langle M\rangle_t - \langle M\rangle_s\right)^{\frac 12 - \delta}$ for $\delta\in (0, \frac 12)$. 
Theorem 5.1 in \cite{ws:BarlowYor} now implies for any $\gamma\ge 1$ that there exists a finite constant $C_{\delta,\gamma}$ such that 
\begin{equation} 
\mathbb{E}\left[ \left( L_{T, \delta} \right)^\gamma \right]^{\frac 1\gamma} \le C_{\delta , \gamma} \mathbb{E}\left[ \langle M\rangle_T^{\delta\gamma} \right]^{\frac 1\gamma} \, . 
\end{equation} 
Combining the last estimate with the previous Theorem \ref{thm:MSE} we obtain for $\gamma\delta \le 1$ that 
\begin{equation}
\label{control:QuadrVariation}
\begin{aligned}  
\mathbb{E}\left[ \left( L_{T, \delta} \right)^\gamma \right]^{\frac 1\gamma} 
\le C_{\delta , \gamma} \mathbb{E}\left[ \langle M\rangle_T^{\delta\gamma} \right]^{\frac 1\gamma}  
\le C_{\delta , \gamma} \mathbb{E}\left[ \langle M\rangle_T \right]^{\delta} 
\le C_{\delta , \gamma} (C_1 + 2)^{\delta} \mathbb{E}\left[ \int_0^T E_t\, {\rm d}t \right]^{\delta} 
\le C \varepsilon^{\frac{\delta}2} 
\end{aligned} 
\end{equation} 
for some constant $C$, depending on $\gamma$, $\delta$, $T$, $C_1$ and on the bound on the mean  
squared error obtained in Theorem \ref{thm:MSE}. We can therefore estimate 
\begin{equation}
\begin{aligned} 
\mathbb{E} & \left[ \sup_{t\le T} e^{-\alpha t} |M_t|+ \alpha \int_0^t e^{-\alpha (t-s)} |M_t - M_s|\, {\rm d}s \right] \\
& \qquad \le \mathbb{E}\left[ \sup_{t\le T} \left( e^{-\alpha t} \langle M\rangle_t^{\frac 12 - \delta}  + \alpha\int_0^t e^{-\alpha (t-s)} \left(\langle M\rangle_t  
- \langle M\rangle_s\right)^{\frac 12 - \delta} \, {\rm d}s \right) L_{T, \delta}\right] \\ 
& \qquad \le (C_1 + 2) \mathbb{E}\left[ \sup_{t\le T} \left( e^{-\alpha t} t^{\frac 12 - \delta} + 
\alpha\int_0^t e^{-\alpha (t-s)} (t-s)^{\frac 12 - \delta} \, {\rm d}s \right) \sup_{t\le T} 
E_t^{\frac 12- \delta} \, L_{T, \delta}\right] \\  
& \qquad \le (C_1 + 2) \frac{\Gamma \left( \frac 32 - \delta \right)}{\alpha^{\frac 12 - \delta}} 
\mathbb{E}\left[ \sup_{t\le T} E_t^{\frac 12- \delta} \, L_{T, \delta}\right]\, . 
\end{aligned} 
\end{equation} 
Applying Young's inequality with $p = \frac 1{\frac 12 - \delta}$ and  
$q = \frac 1{\frac 12 + \delta}$ we can further estimate the right hand side from above by 
\begin{equation}
\begin{aligned} 
(C_1 + 2) \frac{\Gamma \left( \frac 32 - \delta \right)}{\alpha^{\frac 12 - \delta}} 
\mathbb{E}\left[ \sup_{t\le T} E_t^{\frac 12- \delta} L_{T, \delta}\right]
\le \left( \frac 12 - \delta \right)  \mathbb{E}\left[ \sup_{t\le T} E_t \right]  
+ \frac{C}{\alpha^{\frac{1 - 2\delta}{1 + 2\delta}}}  \mathbb{E}\left[ L_{T, \delta}^{\frac 1{\frac 12 + \delta}} 
\right] \, , 
\end{aligned}  
\end{equation} 
for some finite constant $C$ depending on $C_2$ and $\delta$. Taking expectation in 
\eqref{sup_est-omega_wise} and using \eqref{control:QuadrVariation} to estimate the third term 
gives 
\begin{equation}
\begin{aligned}  
\mathbb{E}\left[ \sup_{t\le T}  E_t \right] 
& \le \left( \mathbb{E}\left[E_0\right] 
+ \frac K{\alpha} \right) + \left( \frac 12 - \delta  \right) \mathbb{E} \left[ \sup_{t\le T}  
E_t \right]  + \frac{C}{\alpha^{\frac{1 - 2\delta}{1 + 2\delta}}} 
\mathbb{E}\left[ L_{T, \delta}^{\frac 1{\frac 12 + \delta}} \right] \\
& \le \left( \mathbb{E}\left[E_0\right] 
+ \frac K{\alpha} \right) + \left( \frac 12 - \delta  \right) \mathbb{E} \left[ \sup_{t\le T}  
E_t \right] +  \frac{C}{\alpha^{\frac{1 - 2\delta}{1 + 2\delta}}} 
\varepsilon^{\frac{\delta}{1 + 2\delta}} 
\end{aligned} 
\end{equation} 
with some different constant $C$. Under the assumptions of Theorem \ref{thm:MSE}, in particular 
$\mathbb{E}\left[ E_0\right]\in{\mathcal O}\left( \varepsilon^{\frac 12}\right)$, and thus 
$\alpha^{-1} = {\mathcal O} \left( \varepsilon^{\frac 12} \right)$ for
$\varepsilon\le\varepsilon_0$, $\varepsilon_0$ sufficiently small, we can now find for any 
$\eta \in \left( 0, \frac 14\right)$ now a finite constant $C$ such that 
\begin{equation} 
\mathbb{E}\left[ \sup_{t\le T} E_t\right] \le C \varepsilon^{\frac 12 - \eta} \, . 
\end{equation} 
In particular, 
\begin{equation}
\mathbb{P} \left[ \sup_{t\le T}  E_t \ge c\varepsilon^q \right] \le \frac 1{c\varepsilon^q} 
\mathbb{E}\left[ \sup_{t\le T} E_t\right] = \mathcal{O} \left( \varepsilon^{1/2-\eta -q}\right)\, , 
\end{equation}
which implies that for any $q\in (0, 1/2)$ the estimation error $E_t = \|e_t\|^2/2$ is of order 
$\mathcal{O} \left( \varepsilon^q \right)$ uniformly on $[0,T]$ with probability close to one.

%
\section{Consistency of the ensemble Kalman-Bucy filter for linear systems} \label{sec-linear}
%

 In this section, we provide a detailed analysis of the EnKBF in the case of linear model dynamics, i.e.,~$f(x) = Ax + b$, 
 linear forward map, i.e.~$h(x) = Hx$, full rank diffusion tensor, $D$, and initial ensemble, $X_0^i$, chosen such that $P^M_0$ is invertible.  
 Then the EnKBF (\ref{EKB1b}) reduces to
 \begin{equation} \label{LEKB1a}
{\rm d}X_t^i = (AX_t^i + b) {\rm d}t  + D (P_t^M)^{-1} (X_t^i - \bar x_t^M){\rm d}t - \frac{1}{2} P_t^M H^{\rm T} R^{-1} \left( HX_t^i{\rm d}t + H \bar x_t^M
{\rm d}t - 2{\rm d}Y_t \right)\,,
\end{equation}
$i=1,\ldots,M$, from which we can extract the equation for the empirical mean, $\bar x_t$, 
\begin{equation}
\label{SampleMean}
{\rm d}\bar x_t^M = A\bar x_t^M{\rm d}t + b{\rm d}t - P_t^M H^{\rm T} R^{-1} (H\bar x_t^M{\rm d}t - {\rm d}Y_t)
\end{equation}
and the equation for the empirical covariance matrix, as defined in (\ref{EKB2b}),
\begin{equation}
\label{SampleCovariance}
\frac{{\rm d}}{{\rm d}t} P_t^M = AP_t^M + P_t^M A^{\rm T} +  2D - P_t^M H^{\rm T} R^{-1} H P_t ^M   
\end{equation} 
provided $P_t^M$ has full rank. These equations correspond exactly to the classic Kalman-Bucy filter formulas for the mean and 
the covariance matrix \cite{sr:jazwinski}. However, while one would set $P_0^M$ and $\bar x_0^M$ equal to the mean and the covariance matrix, respectively, 
of the given initial Gaussian distribution ${\rm N}(\bar x_0,P_0)$ in the classic Kalman-Bucy filter formulation, the $P_t^M$ and $\bar x_t^M$ arise in our context 
from sampling from the initial distribution, i.e., $X_0^i \sim {\rm N}(\bar x_0,P_0)$. 

\begin{remark}
It is well-known that solutions to (\ref{SampleCovariance}) have full rank for all $t>0$ even if the initial $P_0^M$ is singular. However, note that 
(\ref{SampleCovariance}) holds true only if $P_0^M$ is non-singular and that the  diffusion induced contribution in (\ref{SampleCovariance})
needs to be replaced by $D(P_t^M)^+ P_t^M$ otherwise. This discrepancy between the Riccati equation for the classic Kalman-Bucy filter 
and the EnKBF is caused by our interacting particle approximation to the diffusion term in (\ref{sde1}).
\end{remark}




\noindent
We will now investigate the asymptotic behavior of the EnKBF in the large ensemble size limit. 
More specifically, we will show 
that the empirical distribution of the EnKBF converges under appropriate conditions
towards a distribution with mean and covariance determined by the Kalman-Bucy filtering equations. 
Note that this does not imply that the empirical distribution of the EnKBF converges to the 
conditional distribution $\pi_t$ given by the solution of the Kushner-Zakai equation \eqref{KS}, 
but by the nonlinear Fokker-Planck equation 
\eqref{limit1} instead as we will show in Section \ref{sec-asymptotic} below. 

Let us first state the following a.s.~result on the asymptotic behavior of $P_t^M$.  

\begin{proposition} 
	\label{PropPathwiseControl}
	Let $\pi_0$ be the initial distribution on $\mathbb R^{N_x}$ with finite second moments and invertible covariance matrix with entries
\begin{equation}	\label{initialcovariancematrix}
	\bar{P}_0 (k,l) =  \pi_0[ x_k x_l] - \pi_0[x_k] \pi_0[x_l] \,,
	\end{equation}
	 $1\le k,l\le N_x$. Let $X^i_0$, $i = 1, 2, \ldots$, be iid ($\pi_0$), and let $\bar{P}_t$ be the solution of the Kalman-Bucy filtering equation 
	 \eqref{KalmanBucyFilteringEq} with initial condition $\bar{P}_0$. Then there exists a constant 
	 \begin{equation}
	 \tilde C = \tilde 
	C(t, A,D, H^T R^{-1}H, \max_{0\le s\le t} \|\bar{P}_s\|_{\rm F} , \sup_{M\ge 2} V^M_0)
	\end{equation} 
	such that  
	\begin{equation}
	\|P_t^M - \bar{P}_t\|_{\rm F}^2 \le e^{t \tilde C} \|P_0^M - \bar{P}_0\|^2_{\rm F} \, , 
	\end{equation}
	where $V_0^M$ is defined by (\ref{l2norm}) with $t=0$.
	\end{proposition} 

\noindent
Note that the strong law of large numbers implies that $\sup_{M\ge 2} V^M_0 < \infty$ $\pi_0$-a.s.  

\begin{proof} 
	Using the dynamical equations \eqref{SampleCovariance} 
	for $P_t^M$ and \eqref{KalmanBucyFilteringEq} for $\bar{P}_t$ (which of course coincides with \eqref{SampleCovariance}), we immediately obtain that 
	\begin{equation}
	\begin{aligned} 
	\frac 12\frac{\rm d}{{\rm d}t} \| P^M_t - \bar{P}_t\|^2_{\rm F} 
	& \le \langle A (P^M_t -\bar{P}_t ), P^M_t - \bar{P}_t\rangle  
	+ \langle \left( P^M_t - \bar{P}_t \right) A^{\rm T},  P^M_t - \bar{P}_t\rangle  \\
	& \qquad  - \langle P_t^M H^{\rm T} R^{-1} HP_t^M - \bar{P}_t H^{\rm T} R^{-1} H\bar{P}_t ,  
	   P^M_t-\bar{P}_t \rangle\, .  
	\end{aligned} 
	\end{equation}
	Using 
    \begin{equation}
	\begin{aligned}
	& \langle P_t^M H^{\rm T} R^{-1} HP_t^M - \bar{P}_t H^{\rm T} R^{-1} H\bar{P}_t , P^M_t-\bar{P}_t 
	\rangle \\
	& \qquad = \langle P_t^M H^{\rm T} R^{-1} H\left( P_t^M - \bar{P}_t \right) , P^M_t - \bar{P}_t 
	  \rangle 
	 + \langle \left( P_t^M - \bar{P}_t\right) H^{\rm T} R^{-1} H\bar{P}_t , P^M_t - \bar{P}_t \rangle \\
	 & \qquad \le \|H^T R^{-1} H\|_{\rm F}  \left(\|P_t^M\|_{\rm F} + \|\bar{P}_t\|_{\rm F} \right)  
	   \|P_t^M - \bar{P}_t\|^2_{\rm F} 	
	\end{aligned} 
	\end{equation}
	we arrive at the following differential inequality 
	\begin{equation}
	\begin{aligned} 
	\frac 12\frac{\rm d}{{\rm d}t} \| P^M_t - \bar{P}_t\|^2_{\rm F}  
	& \le \left( 2\|A\|_{\rm F} + 
	\|H^T R^{-1}H\|_{\rm F} \left(\|P_t^M\|_{\rm F}  
	  + \|\bar{P}_t\|_{\rm F} \right)\right)  \|P_t^M - \bar{P}_t\|^2_{\rm F}\, .  
	\end{aligned} 
	\end{equation}
	Integrating up the last inequality w.r.t.~time $t$ yields 
	\begin{equation} \label{inequality2}
	\begin{aligned} 
	\| P^M_t - \bar{P}_t\|^2_{\rm F}  
	& \le \exp\left( 4t \|A\|_{\rm F}    
	+  \|H^T R^{-1} H\|_{\rm F} \int_0^t \left(\|P_s^M\|_{\rm F} + \|\bar{P}_s\|_{\rm F} \right)\, ds \right)  
	\|P_0^M - \bar{P}_0\|^2_{\rm F} \, .  
	\end{aligned} 
  \end{equation}
    In the next step we will need a uniform in $M$ upper bound on $\|P_t^M\|_{\rm F}$ that holds (locally) uniform w.r.t.~time $t$.  To this end first note that \eqref{UniformMUpperBound} implies   
	\begin{equation}
	\|P_t^M\|_{\rm F} \le V^M_t \le e^{t\|A\|_{\rm F} }\left( V^M_0  + \frac{\mbox{tr}\,(D)}{\|A\|_{\rm F} }\right)\, ,  
	\end{equation}
    thereby using $L_+ \le\|A\|_{\rm F}$. Since the solution $\bar{P}_t$ of \eqref{KalmanBucyFilteringEq} is continuous, hence, also locally bounded, we can estimate the exponential in \eqref{inequality2} from above by  
    $$ 
    2t \left( 2\|A\|_{\rm F} 
    +  \|R^{-1}\|_{\rm F} \|H\|^2_{\rm F} \left( e^{t\,\|A\|_{\rm F} }\left( V^M_0  + \frac{\mbox{tr}\,(D)}{\|A\|_{\rm F}}\right) + \max_{0\le s\le t} \|\bar{P}_s\|_{\rm F} \right) \right) 
    $$ 
    which implies the assertion. 
\end{proof}  

\noindent
We can now state our main result on the asymptotic consistency of the ensemble Kalman filter. 

\begin{theorem} 
	\label{AsymptoticConsistencyEnKF}  
	Suppose that $X^i_0$, $i = 1, 2, 3, \ldots$, are iid ($\pi_0$) where the initial distribution $\pi_0$ has finite second-order moments and 
	invertible covariance matrix (\ref{initialcovariancematrix}). Let $\bar{P}_t$ be the solution of the Kalman-Bucy filtering equation 
	\eqref{KalmanBucyFilteringEq} with initial condition $\bar{P}_0$ and $\bar{x}_t$ be the unique solution of 
	\begin{equation} 
	\label{KalmanBucyMean} 
{\rm d} \bar{x}_t = A\bar{x}_t\, {\rm d}t + b\, {\rm d} t - \bar{P}_t H^T R^{-1} \left( H \bar{x}_t\, {\rm d} t - {\rm d} Y_t \right)
	\end{equation} 
	with initial condition $\bar{x}_0 := \pi_0[x]$. Then $\lim_{M\to\infty} \bar{x}^M_t = \bar{x}_t$  
	in $L^2$, in particular in probability, for all $t\ge 0$. 
\end{theorem} 

\begin{proof} 
	Since $X_0^i$ are iid, the strong law of large numbers implies that $\lim_{M\to\infty} P_0^M = \bar{P}_0$ $\pi_0$-a.s.~and in $L^2$, since $\pi_0$ has finite second moments, thus $\lim_{M\to\infty} P_t^M = \bar{P}_t$ a.s.~and in $L^2$ for $t\ge 0$ due to Proposition \ref{PropPathwiseControl}. 
	
    \medskip 
    \noindent 
    To see that $\bar{x}_t^M$ converges towards the unique solution $\bar{x}_t$ of 
    \eqref{KalmanBucyMean} note that 
    \begin{equation}
    \begin{aligned} 
    {\rm d} \left(\bar{x}^M_t - \bar{x}_t\right)   
    & = A\left( \bar{x}^M_t - \bar{x}_t\right)\, {\rm d} t 
      -  \left( P_t^M H^{\rm T} R^{-1}H\bar{x}^M_t - \bar{P}_t H^{\rm T} R^{-1}H\bar{x}_t \right)\, {\rm d} t  \\ 
    & \qquad + \left( P_t^M - \bar{P}_t \right)H^{\rm T} R^{-1} \, {\rm d} Y_t \, 
    \end{aligned} 
    \end{equation}
    and, consequently, 
    \begin{equation}
    \begin{aligned} 
    \|\bar{x}^M_t - \bar{x}_t\|    
    & \le  \|\bar{x}^M_0 - \bar{x}_0\|  
      + \int_0^t \left( \|A\|_{\rm F} + \|H^T R^{-1}H\|_{\rm F} \|\bar{P}_s\|_{\rm F} \right) 
        \|\bar{x}^M_s - \bar{x}_s\|\, {\rm d} s  \\
    & \qquad +  \int_0^t \|H^T R^{-1} H\|_{\rm F} \|P_s^M- \bar{P}_s\|_{\rm F}  \|\bar{x}_s^M\|\, {\rm d} s 
        + \left\| \int_0^t \left( P_s^M - \bar{P}_s \right)H^{\rm T} R^{-1} \, {\rm d} Y_s\right\| \, . 
    \end{aligned} 
    \end{equation}
    Taking expectations we arrive at 
    \begin{equation}
    \begin{aligned} 
    \mathbb{E} \left[ \|\bar{x}^M_t - \bar{x}_t\|\right]
    & \le  \mathbb{E} \left[ \|\bar{x}^M_0 - \bar{x}_0\|\right] + \int_0^t \left( \|A\|_F   
       + \|H\|^2_{\rm F} \|R^{-1}\|_{\rm F} \|\bar{P}_s\|_{\rm F} \right) 
        \mathbb{E} \left[ \|\bar{x}^M_s - \bar{x}_s\|\right]\, {\rm d} s \\ 
    & \qquad + \int_0^t  \|H\|^2_{\rm F} \|R^{-1}\|_{\rm F}  \mathbb{E} \left[ \|P_s^M- \bar{P}_s\|_{\rm F}\right] 
       \|\bar{x}_s^M\|\, {\rm d} s \\ &\qquad 
       + \mathbb{E} \left[ \left\|\int_0^t \left( P_s^M - \bar{P}_s \right)H^{\rm T} R^{-1} \, {\rm d} Y_s\right\|\right) \, . 
       \end{aligned} 
       \end{equation}
    Using $\lim_{M\to\infty} \mathbb{E}\left[ \|P_t^M - \bar{P}_t\|^2_{\rm F} \right] = 0$ it follows that 
        \begin{equation}
        \lim_{M\to\infty}
   \mathbb{E} \left[ \left\|\int_0^t \left( P_s^M - \bar{P}_s \right)H^{\rm T} R^{-1} \, 
   {\rm d}Y_s\right\|\right] = 0
   \end{equation}
  by dominated convergence, and then Gronwall's lemma implies that 
  $\lim_{M\to\infty} \mathbb{E}\left[ \|\bar{x}^M_t - \bar{x}_t\| \right] = 0$.  
\end{proof} 

\begin{remark}
It is well-known that if $(A,H)$ is observable, i.e., $\mbox{rank }\left[ H^{\rm T} , (HA)^{\rm T}, \ldots , (HA^{N_x-1})^{\rm T}\right] = N_x$, and $(A,C)$ is controllable, i.e., $\mbox{rank }\left[ C , AC, \ldots , A^{N_x -1}C\right] = N_x$, then there exists a unique positive definite solution $P_\infty$ of the matrix Riccati equation 
\begin{equation} 
\label{StationaryKalmanBucyFilter} 
0 = AP_\infty + P_\infty A^{\rm T} + 2D - P_\infty H^{\rm T} R^{-1}H P_\infty\, , 
\end{equation} 
and the solution $P_t$ of the matrix Riccati equation 
\begin{equation} 
\label{KalmanBucyFilteringEq} 
\frac{\rm d}{{\rm d}t} P_t  = AP_t + P_t A^{\rm T} + 2D - P_t H^{\rm T} R^{-1}H P_t\, , 
\end{equation} 
converges for any initial condition $P_0$  towards $P_\infty$ as $t\to\infty$ with exponential rate $\lambda < \lambda_\ast$, where 
\begin{equation}
\lambda_\ast := \inf\{ - Re (\lambda ) \mid \lambda \mbox{ eigenvalue of } A - P_\infty H^{\rm T}R^{-1}  H \}\, .  
\end{equation}
(see \cite{ws:KwakernaakSivan}, Theorem 4.11, and \cite{ws:OconePardoux}, Lemma 2.2).

Now recall that we have assumed in Sections \ref{sec-problem-well-posedness} and \ref{sec-accuracy} that $h(x) = x$, i.e.~$H=I$, and 
that $D = CC^T$ has full rank. In other words, we have assumed a restricted case of (nonlinear) controllability and observability. It would be of interest to explore in as far the conditions of
Sections \ref{sec-problem-well-posedness} and \ref{sec-accuracy} can be relaxed while maintaining the well-posedness, stability and accuracy of the
associated EnKBF.
\end{remark}

%
\section{Asymptotic limiting equations for the extended EnKBF} \label{sec-asymptotic}
%

In this section, we will derive the non-Markovian stochastic differential equation (\ref{EKB1a}) 
with (\ref{EKB2a}) of McKean-Vlasov type. 
We first have to show now that \eqref{EKB1a} is well-posed. To this end we assume that 
$f$, $h$ are globally Lipschitz continuous and that the initial condition $\hat{X}_0$ has finite second 
moments with invertible covariance matrix ${\cal P}_0$. 
Recall that - given $X_t = X^{\rm ref}_t$ - the observation process $Y_t$ can be interpreted as 
Brownian motion with covariance operator $R$ and drift term $h(X^{\rm ref}_t )$, so that we can 
solve \eqref{EKB1a} uniquely up to the first time $\tau$ where ${\cal P}_\tau$ becomes singular. 
Clearly, $\tau > 0$ a.s. (w.r.t.~the distribution of 
$\{Y_s\}$). Using It\^o's formula, it is then straightforward to see that the distribution 
$\hat{\pi}_t$ of $\hat{X}_t$ indeed satisfies the nonlinear Fokker-Planck equation 
\eqref{limit1} (up to time $\tau$).

\subsection{Lower bounds on $\lambda^{\rm min} (\mathcal P_t)$ and well-posedness of \eqref{EKB1a}} 

We will prove in the Lemma \ref{lemma3} below a strictly positive lower bound on the smallest 
eigenvalue $\lambda^{\rm min} ({\cal P}_t)$  of ${\cal P}_t$ locally uniformly w.r.t.~$t$, 
a.s.~w.r.t.~the distribution of $\{Y_s\}$, under appropriate assumptions on the coefficients 
$f,h, D$ and $R$. This implies in particular that ${\cal P}_t$ will stay invertible for all $t$, 
a.s. and yields existence and uniqueness of a strong solution of \eqref{EKB1a} for all times $t$ 
(for typical observation $\{Y_s\}$). On the other hand, using the algebraic identity 
\begin{equation} 
(P_s^{M})^{-1} - {\cal P}_s^{-1} = (P_s^{M})^{-1} \left( {\cal P}_s - P_s^{M}\right) 
{\cal P}_s^{-1}  
\end{equation}
we also obtain the following control 
\begin{equation}
\label{UnifControlInverse} 
\|(P_s^{M})^{-1} - {\cal P}_s^{-1}\|_2 \le  C(t)^2 \|{\cal P}_s - P_s^{M}\|_2  \, , \quad s\le t\, , 
\end{equation} 
for the distance between the inverse covariance matrix of the EnKBF and ${\mathcal P}_t$. Here, 
$C(t)$ is a joint upper bound of $\|{\cal P}_s^{-1}\|_2$ and $\|(P_s^{M})^{-1}\|_2$ (uniform in 
$M$) for $s\le t$.

To this end let us first state the dynamical equations for the mean $\bar{x}_t$ and the covariance 
matrix ${\cal P}_t$ (analogous to \eqref{EKBm} and \eqref{EKBP} for the EnKBF): 
\begin{equation} 
\label{MFm}
d\bar{x}_t = \bar{f}_t\, dt - {\cal Q}_t R^{-1} \left( \bar{h}_t\, dt - dY_t\right) 
\, , \quad  t< \tau\,,
\end{equation} 
with $\bar{f}_t = \mathbb{E}\left[ f(\hat{X}_t)\right]$ and 
\begin{equation} 
\label{MFP} 
\frac{d}{dt}{\cal P}_t = \mathbb{E}\left[ (f(\hat{X}_t ) - \bar{f}_t) 
(\hat{X}_t - \bar{x}_t)^{\rm T}+ (\hat{X}_t - \bar{x}_t)(f(\hat{X}_t ) - \bar{f}_t)^{\rm T}\right] 
+ 2D - {\cal Q}_t R^{-1}{\cal Q}^{\rm T}_t \, ,\quad t< \tau\, . 
\end{equation}



\begin{lemma} 
\label{lemma1} 
\begin{equation}
\frac 1{\sqrt{N_x}} \mathbb{E}\left[ \|\hat{X}_t - \bar{x}_t\|^2\right]
\le \| {\cal P}_t \|_{\rm F} \le \mathbb{E}\left[ \|\hat{X}_t - \bar{x}_t\|^2\right] 
\, ,\quad  t \le \tau  \, .  
\end{equation}
\end{lemma} 

\begin{proof} 
Similar to the proof of Lemma \ref{ControlFrobeniusNorm}:   

\medskip 
\noindent 
Upper bound: 
\begin{equation}
\begin{aligned} 
\|{\cal P}_t\|^2_{\rm F} & = \sum_{k,l} \mathbb{E}\left[ \left(\hat{X}_t - \bar{x}_t\right)(k) \left( \hat{X}_t - \bar{x}_t\right)(l) \right]^2 \\
& \le \sum_{k,l} \mathbb{E}\left[ \left(\hat{X}_t - \bar{x}_t\right)^2(k)\right] 
\mathbb{E}\left[ \left( \hat{X}_t - \bar{x}_t\right)^2(l) \right]
= \mathbb{E}\left[\|\hat{X}_t - \bar{x}_t\|^2\right]^2 \, . 
\end{aligned} 
\end{equation}

\noindent 
Lower bound: 
\begin{equation}
\begin{aligned} 
\|{\cal P}_t\|^2_{\rm F}  = \sum_{k,l} \mathbb{E}\left[ \left(\hat{X}_t - \bar{x}_t\right)(k) 
\left( \hat{X}_t - \bar{x}_t\right)(l) \right]^2 
 \ge \sum_k \mathbb{E}\left[ \left(\hat{X}_t - \bar{x}_t\right)^2(k)\right]^2 \,.
\end{aligned} 
\end{equation}
\end{proof}

\begin{lemma} 
\label{lemma2} 
For all $t < \tau$ there exists some finite constant $C_{4} (t)$ - independent of $\{Y_s\}$ - such that 
\begin{equation}
\label{lemma2:eq1}
\sup_{0\le s\le t} \mathbb{E}\left[ \|\hat{X}_s - \bar{x}_s\|^2\right] \le C_{4} (t) \, . 
\end{equation}
\end{lemma} 

\begin{proof}  
The difference $\hat{X}_t - \bar{x}_t$ satisfies the ordinary differential equation 
\begin{equation}
\frac{\rm d}{{\rm d} t}\left( \hat{X}_t - \bar{x}_t\right) 
= \left (f(\hat{X}_t) - \bar{f}_t\right)  + D{\cal P}_t^{-1} \left( \hat{X}_t - \bar{x}_t\right) - \frac 12 {\cal Q}_t R^{-1} \left( h(\hat{X}_t) - \bar{h}_t\right)  
\end{equation}
up to time $\tau$ so that for $t <  \tau$ 
\begin{equation}
\label{L2NormMFProcess} 
\begin{aligned} 
\frac {\rm d}{{\rm d} t} \mathbb{E}\left[ \|\hat{X}_t - \bar{x}_t\|^2\right] 
& = 2\mathbb{E}\left[ \langle f(\hat{X}_t) - \bar{x}_t , \hat{X}_t - \bar{x}_t \rangle \right] 
 + 2\mathbb{E}\left[ \langle D{\cal P}_t^{-1}\left( \hat{X}_t - \bar{x}_t\right) , \hat{X}_t - \bar{x}_t 
   \rangle \right] \\ 
   & \qquad - \mathbb{E}\left[ \langle {\cal Q}_t R^{-1} \left( h(\hat{X}_t) - \bar{h}_t\right), 
       \hat{X}_t - \bar{x}_t\rangle \right] \\ 
& \le 2 L_+ \mathbb{E}\left[\|\hat{X}_t - \bar{x}_t\|^2\right] + 2 \mbox{ tr}\,(D) 
\end{aligned} 
\end{equation}
thereby using 
\begin{equation}
\mathbb{E}\left[ \langle {\cal Q}_t R^{-1} \left( h(\hat{X}_t) - \bar{h}_t\right), 
       \hat{X}_t - \bar{x}_t\rangle \right] = \|R^{-1/2}{\cal Q}_t^{\rm T}\|_{\rm F}^2 \ge 0\, . 
\end{equation}
This implies the same bound  
\begin{equation}
\mbox{Var}\,(\hat{X}_t) := \mathbb{E}\left[ \|\hat{X}_t - \bar{x}_t\|^2\right] 
\le e^{2L_+ t} \left( \mathbb{E}\left[ \|\hat{X}_0 - \bar{x}_0\|^2\right] + \frac{\mbox{tr}\,(D)}{L_+}\right)  
\end{equation}
as stated in Remark \ref{UnifMBds} for the EnKBF for $h(x) = x$, therefore, 
\begin{equation} 
\sup_{0\le s\le t} \mbox{Var}\,(\hat{X}_s) 
= \sup_{0\le s\le t} \mathbb{E}\left[ \|\hat{X}_s - \bar{x}_s\|^2\right] 
\le C_{4} (t)
\end{equation}
for some finite constant $C_{4} (t)$ depending on $t$. Note that 
$C_{4} (t)$ clearly is independent of $\{Y_s\}$. 
\end{proof}

\begin{lemma} 
\label{lemma3} 
Let $\|f\|_{\rm Lip}^2 < 2\lambda^{\min} (D) \|R^{-1}\|_{\rm F} \|h\|^2_{\rm Lip}$. If 
\begin{equation}
\lambda^{\rm min} ({\cal P}_0 ) \ge \kappa_- :=  
\frac{2\lambda^{\min} (D)\|R^{-1}\|_{\rm F} \|h\|^2_{\rm Lip} - \|f\|_{\rm Lip}^2} 
{2\|R^{-1}\|_{\rm F}^2 \|h\|^4_{\rm Lip} C_{4} (t)}\, , 
\end{equation}
where $C_{4} (t)$ is the upper bound (\ref{lemma2:eq1}) obtained in the previous 
Lemma \ref{lemma2}, then $\lambda^{\rm min} ({\cal P}_s)\ge\kappa_-$ for all $s < \tau\wedge t$. In particular, 
$\tau > t$. 
\end{lemma}

\begin{proof}  
We will use the representation $\lambda^{\rm min}({\cal P}_t) = \inf_{\|v\|=1} 
\langle {\cal P}_t v, v\rangle$. So fix $v$ with $\|v\| = 1$. Then  
\begin{equation}
\label{SpectralRadiusMFProcess} 
\frac {\rm d}{{\rm d}t} \langle {\cal P}_t v, v\rangle = 2\mathbb{E}\left[ \langle f(\hat{X}_t) 
- \bar f_t, v\rangle \langle\hat{X}_t - \bar{x}_t , v\rangle \right] + 2\langle Dv,v\rangle  - \langle R^{-1} {\cal Q}^{\rm T}_t v, {\cal Q}^{\rm T}_t v\rangle\, . 
\end{equation}  
Using 
\begin{equation}
\langle {\cal P}_t v,v\rangle = \mathbb{E}\left[ \langle \hat{X}_t - \bar{x}_t , v\rangle^2\right] 
\end{equation}
and 
\begin{equation}
\begin{aligned}  
\langle R^{-1} {\cal Q}^{\rm T}_t v, {\cal Q}_t^{\rm T} v\rangle 
& = \langle R^{-1}  \mathbb{E}\left[ (h(\hat{X}_t) - \bar{h}_t)\langle \hat{X}_t - \bar{x}_t, v\rangle\right],
 \mathbb{E}\left[ (h(\hat{X}_t) - \bar{h}_t)\langle \hat{X}_t - \bar{x}_t, v\rangle\right] \rangle \\ 
& \le \|R^{-1}\|_{\rm F} \mathbb{E}\left[ \| h(\hat{X}_t) - \bar{h}_t\|^2 \right] 
\mathbb{E}\left[\langle \hat{X} - \bar{x}_t , v\rangle^2 \right] \\
& \le  \|R^{-1}\|_{\rm F} \|h\|^2_{\rm Lip} \mbox{ Var}\left( \hat{X}_t\right) \langle {\cal P}_t v,v\rangle \, ,
\end{aligned} 
\end{equation}
we can estimate 
\begin{equation}
\begin{aligned}  
\frac {\rm d}{{\rm d}t} \langle {\cal P}_t v, v\rangle 
& \ge -2 \|f\|_{\rm Lip} \mbox{ Var} (\hat{X}_t)^{\frac 12} \|v\| \langle {\cal P}_t v,v\rangle^{\frac 12} 
+ 2\langle Dv,v\rangle - \|h\|^2_{\rm Lip} \|R^{-1}\|_{\rm F} \mbox{Var}(\hat{X}_t) \langle {\cal P}_t v,v\rangle \\ 
& \ge -2 \|f\|_{\rm Lip} C_{4}(t)^{1/2} \langle {\cal P}_t v,v\rangle^{\frac 12} 
+ 2\langle Dv,v\rangle - \|h\|^2_{\rm Lip} \|R^{-1}\|_{\rm F} C_{4} (t) \langle {\cal P}_t v,v\rangle \\ 
& \ge 2\lambda^{\min} (D) - \frac{\|f\|_{\rm Lip}^2}{\|R^{-1}\|_{\rm F} \|h\|^2_{\rm Lip}} -2 \|h\|^2_{\rm Lip} \|R^{-1}\|_{\rm F} C_{4} (t)\langle {\cal P}_t v,v\rangle \, .     
\end{aligned} 
\end{equation}
Now $\lambda^{\rm min} ({\cal P}_0 )\ge  \kappa_-$ implies that 
$\langle {\cal P}_0 v,v\rangle\ge \kappa_-$ and thus $\langle {\cal P}_s v,v\rangle\ge \kappa_-$ 
for all $s< \tau\wedge t$. Hence $\lambda^{\rm min} ({\cal P}_s ) \ge\kappa_- > 0$ for all 
$s < \tau\wedge t$ so that $\tau > t$, since otherwise $\lim_{s\uparrow\tau} 
\lambda^{\rm min} ({\cal P}_s ) = 0$. 
\end{proof} 

\noindent 
The lower bound on $\lambda^{\rm min} ({\cal P}_t )$, locally uniformly w.r.t.~$t$, implies that 
the coefficients of \eqref{EKB1a} are globally Lipschitz on bounded time-intervals, which gives 
existence and uniqueness of strong solutions by standard results for all $t$, a.s. (w.r.t.~the 
distribution of $\{Y_s\}$). 

\subsection{Convergence of the extended EnKBF to the solution of \eqref{EKB1a}} 

We are now ready to state our main result on the asymptotic behavior of the extended EnKBF: 

\begin{theorem} 
\label{MFConvergence} 
Assume that $\|f\|_{\rm Lip}^2 < 2\lambda^{\min} (D) \|R^{-1}\|_{\rm F} \|h\|^2_{\rm Lip}$. Let $\pi_0$ be a 
distribution on $\mathbb R^{N_x}$ with finite support and invertible covariance matrix 
${\cal P}_0$ satisfying $\lambda^{\min} ({\cal P}_0)\ge \kappa_-$, where $\kappa_-$ is as in 
Lemma \ref{lemma3}. Let $\hat{X}_t^i$ be solutions of the mean-field 
process \eqref{EKB1a} with initial conditions $\hat{X}^i_0 = X_0^i$ and $X_0^i$ are iid 
$(\pi_0 )$, so that the solutions $\hat{X}_t^i$ to the mean field processes are iid too. Then 
\begin{equation}
\lim_{M\to\infty} \mathbb{E}\left[\frac 1M \sum_{i=1}^M \|X^i_t - \hat{X}_t^i\|^2  \right]  
= 0 \, . 
\end{equation} 
In particular, 
\begin{equation}
\lim_{M\to\infty} \frac 1M \sum_{i=1}^M g(X_t^i) - \hat{\pi}_t [g] = 0 
\end{equation}
in $L^2 (\mathbb{P})$, hence in probability, for any Lipschitz continuous function $g$. Here, the 
expectation is taken also w.r.t.~the distribution of $\{ Y_s\}$. 
\end{theorem}

\begin{remark} 
The last theorem implies by general theory that the empirical distribution $\hat{\pi}_t^M$, 
defined in (\ref{EmpiricalDistribution}), of the extended EnKBF with $M$ ensemble members 
converges weakly towards the distribution $\hat{\pi}_t$ of the mean 
field process \eqref{EKB1a} in probability w.r.t.~the distribution of $\{Y_s\}$. 
\end{remark}

\begin{remark} The conditions of Theorem \ref{MFConvergence} are satisfied for fully observed 
processes $h(x) = x$, measurement error covariance matrix $R= \varepsilon I$, $\varepsilon >0$ 
sufficiently small, and full rank diffusion tensor $D$, i.e., for the filtering setting considered 
in Sections \ref{sec-problem-well-posedness} and \ref{sec-accuracy}.
\end{remark}

\begin{proof} (of Theorem \ref{MFConvergence}) It\^o's formula implies that 
\begin{equation}
\label{MFConvergence:eq1}
\begin{aligned} 
{\rm d} \left( \frac 1M \sum_{i=1}^M \|\Delta {X}_t^i\|^2 \right) 
& = \frac 2M \sum_{i=1}^M \langle f(X_t^i) - f(\hat{X}_t^i), \Delta X_t^i \rangle \, {\rm d} t \\
& \qquad + \frac 2M \sum_{i=1}^M \langle D\left( (P_t^{M})^{-1} \left( X_t^i - \bar{x}^M_t \right) - 
\mathcal{P}_t^{-1} \left( \hat{X}_t^i - \bar{x}_t\right)\right), \Delta X_t^i \rangle\, {\rm d} t \\ 
& \qquad - \frac 1M \sum_{i=1}^M \langle Q^M_t R^{-1} \left( h(X_t^i) + \bar{h}^M_t\right) - 
\mathcal{Q}_t R^{-1} \left( h(\hat{X}_t^i) + \bar{h}_t\right) , \Delta X_t^i \rangle\, {\rm d} t \\
& \qquad  + \frac 2M \sum_{i=1}^M \langle  \left( Q^M_t - \mathcal{Q}_t \right)
R^{-1}\, {\rm d} Y_t, \Delta X^i_t ,\rangle \\
& \qquad 
+ \frac 1M \mbox{tr}\left( \left(Q_t^M - \mathcal{Q}_t\right)R^{-1}\left( Q_t^M - 
\mathcal{Q}_t\right)^{\rm T} \right)\, {\rm d} t  \\ 
& = I + \ldots + V\, ,
\end{aligned} 
\end{equation} 
with the abbreviation $\Delta X_t^i = X_t^i - \hat X_t^i$.
Our aim is to estimate the right hand side of \eqref{MFConvergence:eq1} in terms of 
$\frac 1M \sum_{i=1}^M \|X_t^i -\hat{X}_t^i\|^2$ and then to apply the Gronwall inequality. This 
requires in particular to control the stochastic integral $IV$ w.r.t.~the observation $\{Y_s\}$. 
Using the decomposition 
${\rm d}Y_t = h(X_t^{\rm ref})\, {\rm d}t + R^{1/2}{\rm d} B_t$ we can split up the stochastic integral 
$IV$ into 
\begin{equation}
\begin{aligned} 
\frac 2M \sum_{i=1}^M \langle  \left( Q^M_t - \mathcal{Q}_t \right)
R^{-1}{\rm d}Y_t, \Delta X^i_t \rangle
& = \frac 2M \sum_{i=1}^M \langle  \left( Q^M_t - \mathcal{Q}_t \right)
R^{-1}h\left( X_t^{\rm ref}\right), \Delta X^i_t \rangle {\rm d}t \\ 
& \qquad + \frac 2M \sum_{i=1}^M \langle  \left( Q^M_t - \mathcal{Q}_t \right)
R^{-1/2}{\rm d}B_t, \Delta X^i_t \rangle \\
& = IVa + IVb\, . 
\end{aligned} 
\end{equation} 
We can now estimate the right hand side of the above equation for $t\le T$ from above as follows 
\begin{equation}
\begin{aligned}
{\rm d}\left( \frac 1M \sum_{i=1}^M \|X_t^i -\hat{X}_t^i\|^2 \right) 
&\le U_M(t)\left( \frac 1M \sum_{i=1}^M \|X_t^i - \hat{X}_t^i\|^2  + R_M (t)\right) {\rm d} t \\
& \qquad 
+ \frac 2M \sum_{i=1}^M \langle  \left( Q^M_t - \mathcal{Q}_t \right)\, 
R^{-1/2} {\rm d}B_t, X^i_t - \hat{X}_t^i\rangle 
\end{aligned}
\end{equation}
thereby keeping the stochastic integral $IVb$. Here, 
\begin{equation}
\begin{aligned} 
U_M (t) & = C_T\left( 1 + \left\| h\left( X^{\rm ref}_{0:T}\right)\right\|_\infty^2 
 + \frac 1M \sum_{i=1}^M \|\hat{X}_t^i\|^2 \right) \\ 
& \qquad\qquad 
\left( 1 + \frac 1M \sum_{i=1}^M \|X_t^i- \bar{x}_t^M\|^2 + \frac 1M \sum_{i=1}^M \|\hat{X}_t^i- 
\bar{x}_t\|^2\right)\, ,  
\end{aligned} 
\end{equation}
with some finite constant $C_T$, and a remainder $R_M (t)$ that converges to zero in 
$L^p (\mathbb{P})$ as $M\to\infty$ for all finite $p$.

Indeed, this is obvious for term $I$, using that $f$ is globally Lipschitz, for terms $III$, $IVa$ and $V$ 
using \eqref{lemma4:eq2} in Lemma \ref{lemma4} in the Appendix and for term $II$ it follows from 
\eqref{lemma4:eq1} in Lemma \ref{lemma4} in the Appendix in combination with 
\eqref{UnifControlInverse}. 

Applying It\^o's product formula to the process 
$e^{-\int_0^t U_M (s){\rm d}s}\frac 1M \sum_{i=1}^M \|X_t^i -\hat{X}_t^i\|^2$ and taking expectations w.r.t.~the distribution of $\{Y_s\}$, we arrive at the following estimate 
\begin{equation}
\begin{aligned} 
\mathbb{E}\left[ e^{-\int_0^t U_M (s){\rm d}s} \frac 1M \sum_{i=1}^M \|X_t^i -\hat{X}_t^i\|^2 \right]
&  \le C_T \mathbb{E}\left[ \int_0^t e^{-\int_0^s U_M (r){\rm d}r} U_M (s) R_M (s){\rm d}s \right]     
\end{aligned}
\end{equation} 
for $t\le T$. Since $U_M R_M$ is bounded by some finite constant plus some power of $ \frac 1M \sum_{i=1}^M \|\hat{X}_t^i\|^2$ and the latter one has some finite exponential 
moment by Lemma \ref{lemma3a} below, it follows that 
\begin{equation}
\lim_{M\to\infty} \mathbb{E} 
\left[ e^{-\alpha_T \int_0^t  \frac 1M \sum_{i=1}^M \|\hat{X}_s^i\|^2 {\rm d}s} 
\frac 1M \sum_{i=1}^M \|X_t^i -\hat{X}_t^i\|^2 \right] = 0 \, , \quad t\le T\, ,    
\end{equation}
for some $\alpha_T > 0$. Now, using Lemma \ref{lemma3a} again, we also may now conclude that 
\begin{equation}
\begin{aligned} 
\lim_{M\to\infty} \mathbb{E}\left[ \frac 1M \sum_{i=1}^M \|X_t^i -\hat{X}_t^i\| \right]^2  
& \le \sup_{M\ge 2} \mathbb{E}\left[ e^{\alpha_T \int_0^t  \frac 1M 
\sum_{i=1}^M \|\hat{X}_s^i\|^2 {\rm d}s}\right]  \\ 
& \qquad \times \lim_{M\to\infty} \mathbb{E}\left[ e^{-\alpha_T \int_0^t  
\frac 1M \sum_{i=1}^M \|\hat{X}_s^i\|^2 {\rm d}s} 
\frac 1M \sum_{i=1}^M \|X_t^i -\hat{X}_t^i\|^2 \right] = 0 \,,
 \end{aligned} 
\end{equation}
for all $t\le T$.  
\end{proof}

%
\section{Numerical example} \label{sec-numerics}
%

\begin{figure}
\begin{center}
\includegraphics[width=0.45\textwidth]{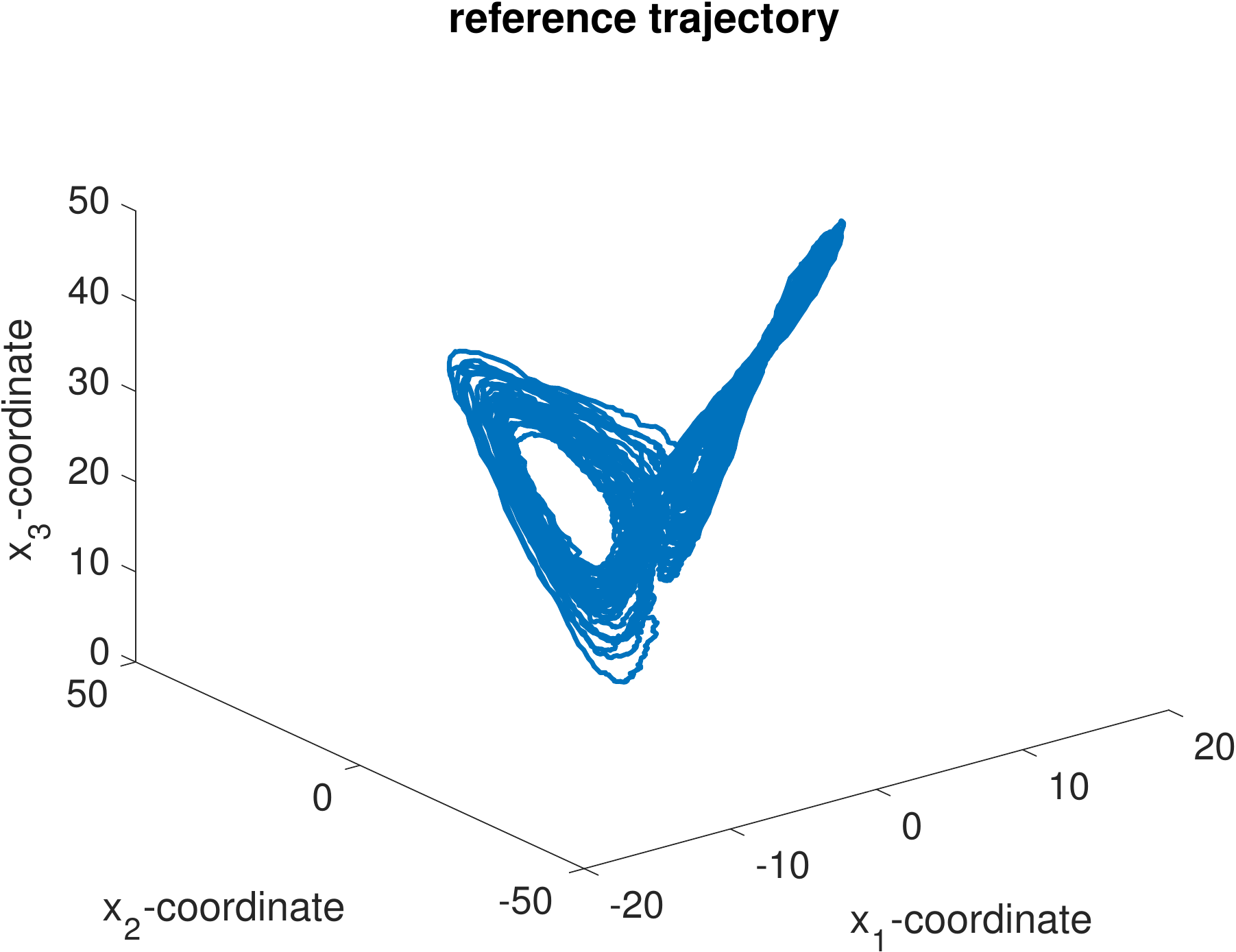} $\quad$ 
\includegraphics[trim={0 6cm 0 6cm},clip,width=0.45\textwidth]{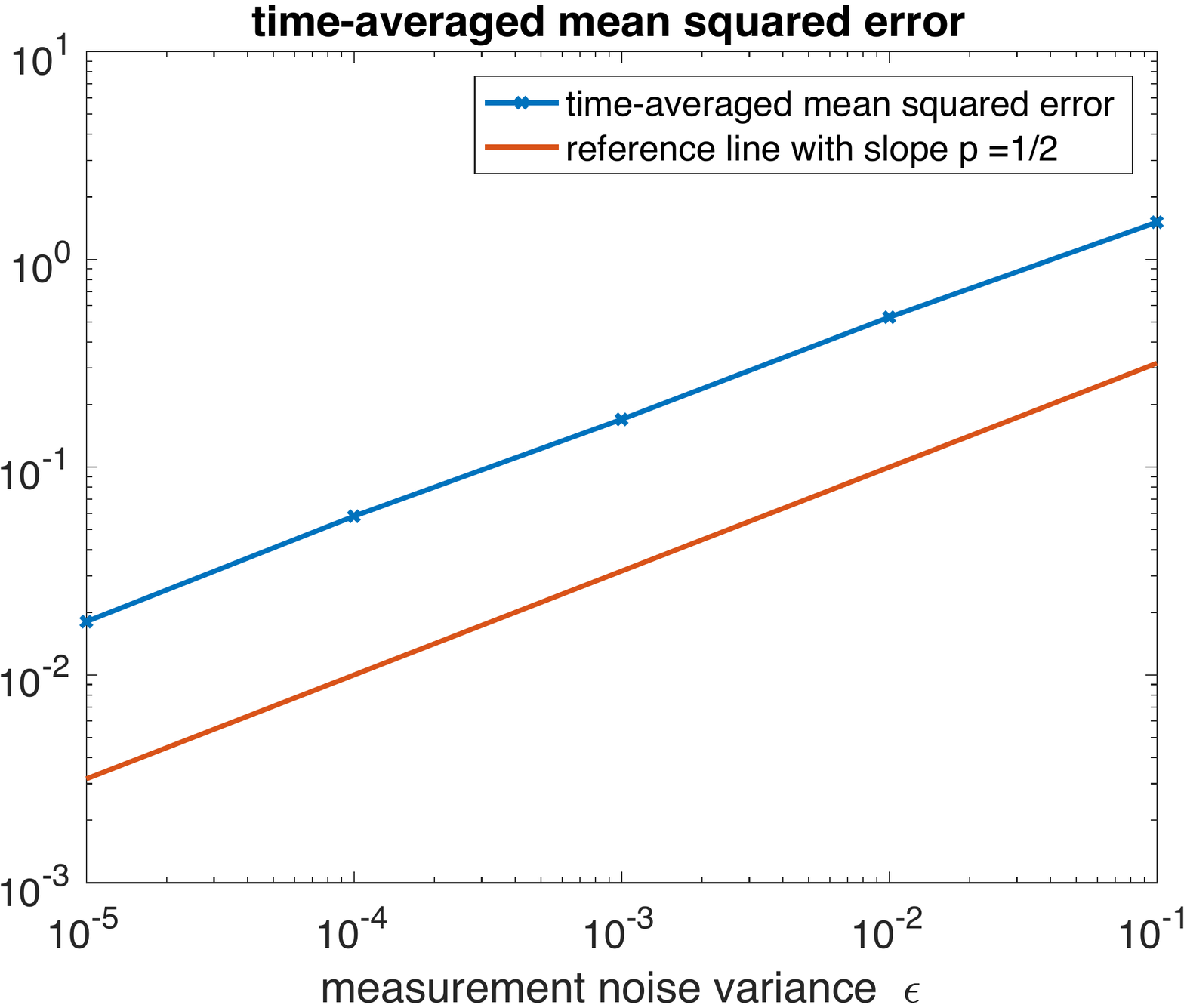}
\end{center}
\caption{Reference trajectory (left panel) and time-averaged mean squared error as a function of the measurement 
error variance $\varepsilon$ (right panel).} 
\label{figure1}
\end{figure}

\begin{figure}
\begin{center}
\includegraphics[trim={0 5.5cm 0 5.5cm},clip,width=0.45\textwidth]{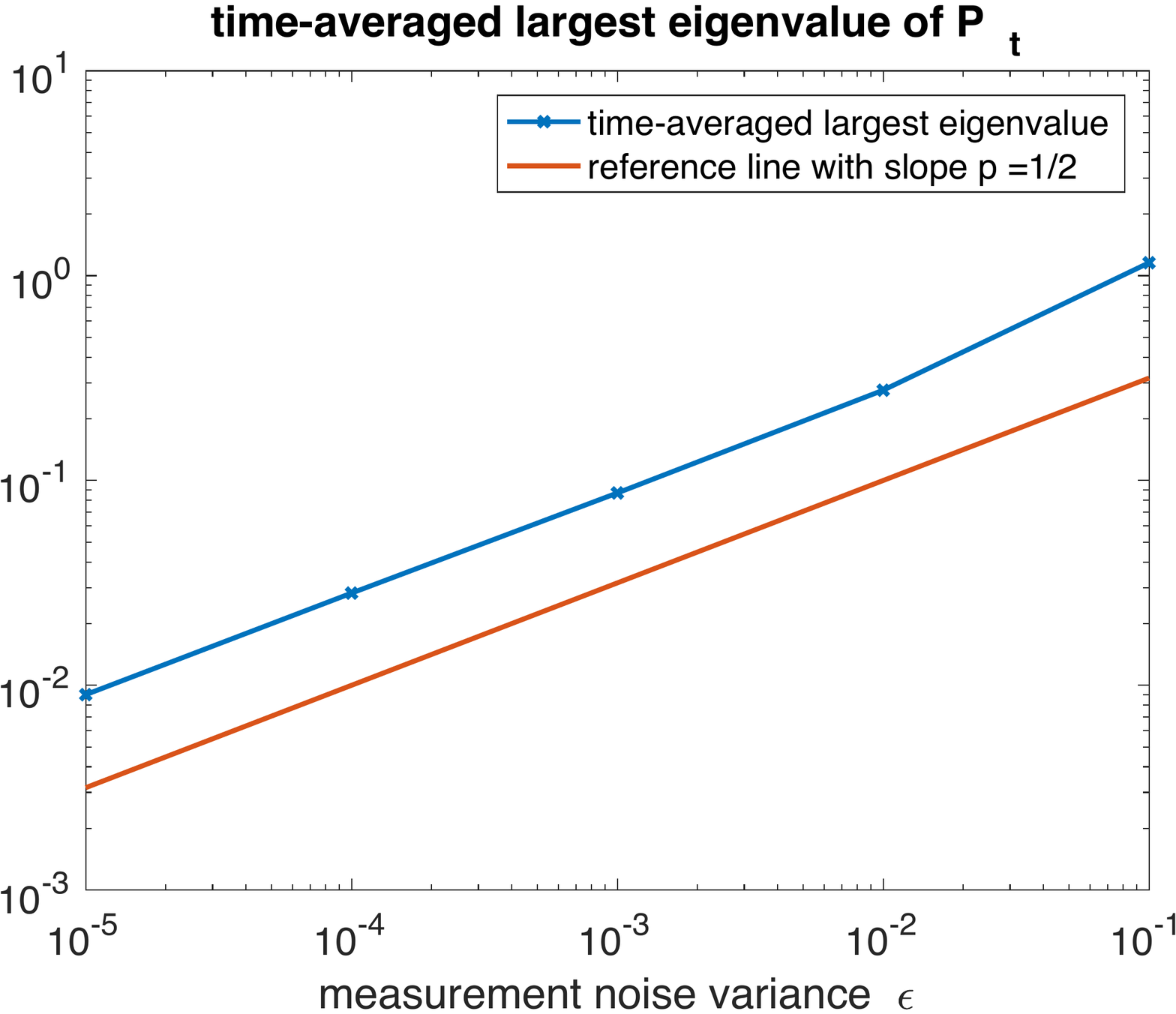} $\quad$
\includegraphics[trim={0 5.5cm 0 5.5cm},clip,width=0.45\textwidth]{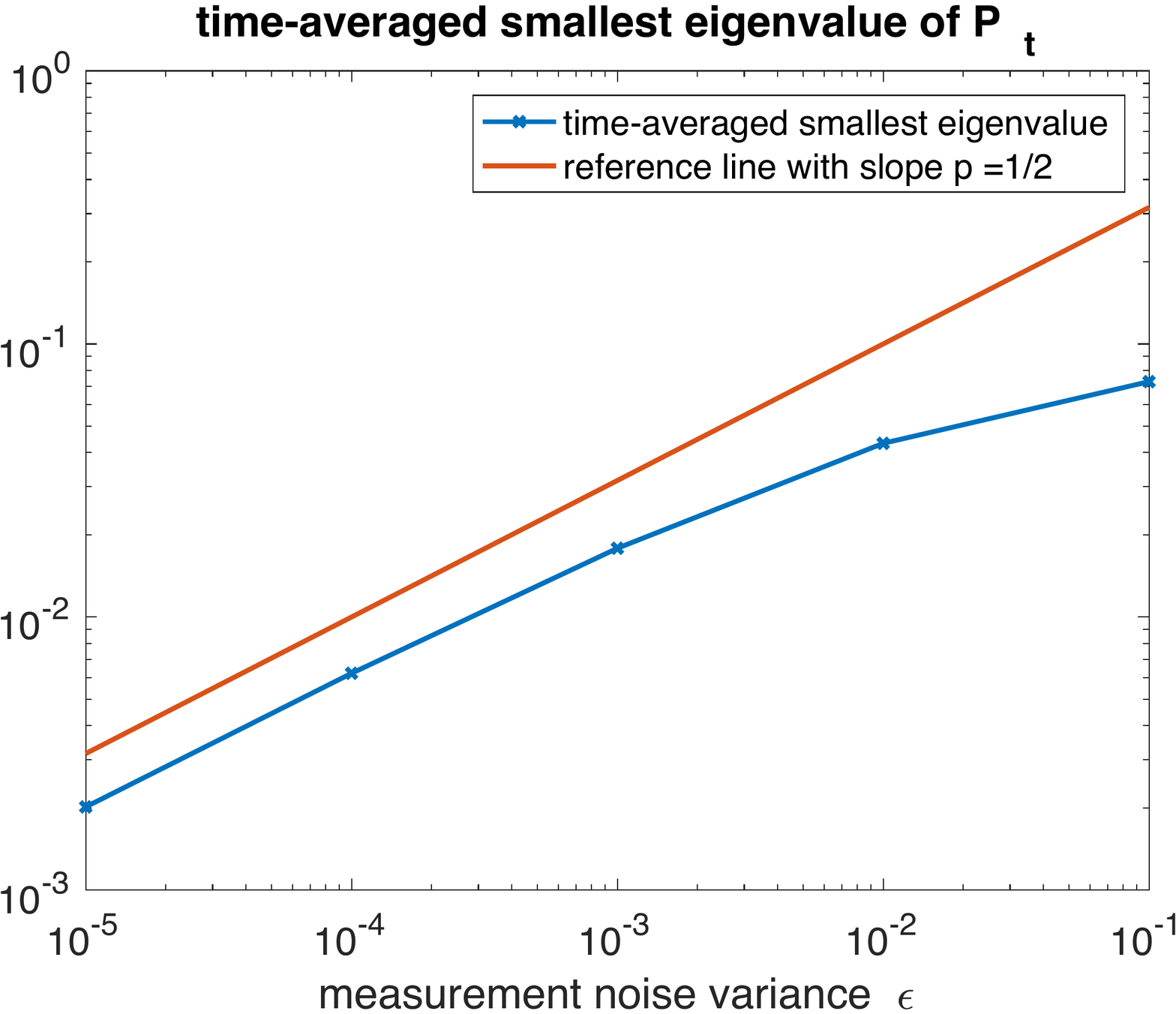} 
\end{center}
\caption{Time-averaged largest (left panel) and smallest (right panel) eigenvalues of $P_t$ as a function of the
measurement error variance $\varepsilon$} \label{figure2}
\end{figure}

We consider the stochastically perturbed Lorenz-63 system \cite{sr:lorenz63,sr:stuart15}, which 
leads to $N_x = 3$, $D=C = I_3$, and drift term given by
\begin{equation}\label{L63_f}
f(x) = \left( \begin{array}{l} 10(x_2-x_1) \\ (28-x_3)x_1 - x_2 \\ x_1 x_2 - \frac{8}{3} x_3 \end{array} \right),
\end{equation}
where $x = (x_1,x_2,x_3)^{\rm T}$. Solutions of the Lorenz-63 system diverge exponentially fast and
filtering is required in order to track a reference solution. Although (\ref{L63_f}) is only locally Lipschitz continuous, the results from this paper are likely to be applicable to the Lorenz-63 system due to 
the existence of a Lyapunov function. 

We apply the EnKBF with ensemble size $M=4$
for values of the measurement error variances $\varepsilon \in \{ 10^{-1},\ldots,10^{-4},10^{-5}\}$. 
The stochastic evolution equations of the EnKBF are solved by the following modified 
Euler-Maruyama scheme 
\begin{equation}
X^i_{n+1} = X_n^i + \Delta t f(X_n^i) + \Delta t (P^M_n)^{-1}(X_n^i-\bar{x}^M_n) -
\frac{1}{2} P_n^M \left(P_n^M + \frac{\varepsilon}{\Delta t} I_3 \right)^{-1} \left(
X_n^i + \bar{x}_n^i - 2\frac{\Delta Y_n}{\Delta t} \right)
\end{equation}
with step-size $\Delta t = 0.00005$ over a total of $10^7$ time-steps. Note that
\begin{equation}
\left(P_n^M + \frac{\varepsilon}{\Delta t} I_3 \right)^{-1}  \approx \frac{\Delta t}{\epsilon} I_3
\end{equation}
for $\Delta t$ sufficiently small and the modification is introduced for numerical stability reasons.
See \cite{sr:akir11} for more details.

The results can be found in Figures \ref{figure1} and \ref{figure2}.  The numerical results are in agreement with our theoretical findings, which predicted an 
${\cal O}(\varepsilon^{1/2})$ behavior of these quantities. While this scaling holds for the 
time-averaged mean squared error and the time-averaged largest eigenvalue of 
$P_t^M$ for the whole range of considered values of $\varepsilon$, the time-averaged smallest eigenvalue truncates slightly off for the larger values of $\varepsilon$. We can also see that there is a gap between the smallest and largest eigenvalues of $P_t^M$ on average.

We repeated the experiment for ensemble sizes of $M=2$ and $M=3$, in which case $P_t^M$ is singular. We still find that the time-averaged mean squared error is roughly of ${\cal O}(\varepsilon^{1/2})$. See Figure \ref{figure3}. The results are in line with those obtained 
in \cite{sr:hunt13} for hyperbolic dynamical systems. We will further investigate the theoretical 
properties of the EnKBF under singular $P_t^M$ in a separate paper. 

\begin{figure}
\begin{center}
\includegraphics[trim={0 6cm 0 6cm},clip,width=0.45\textwidth]{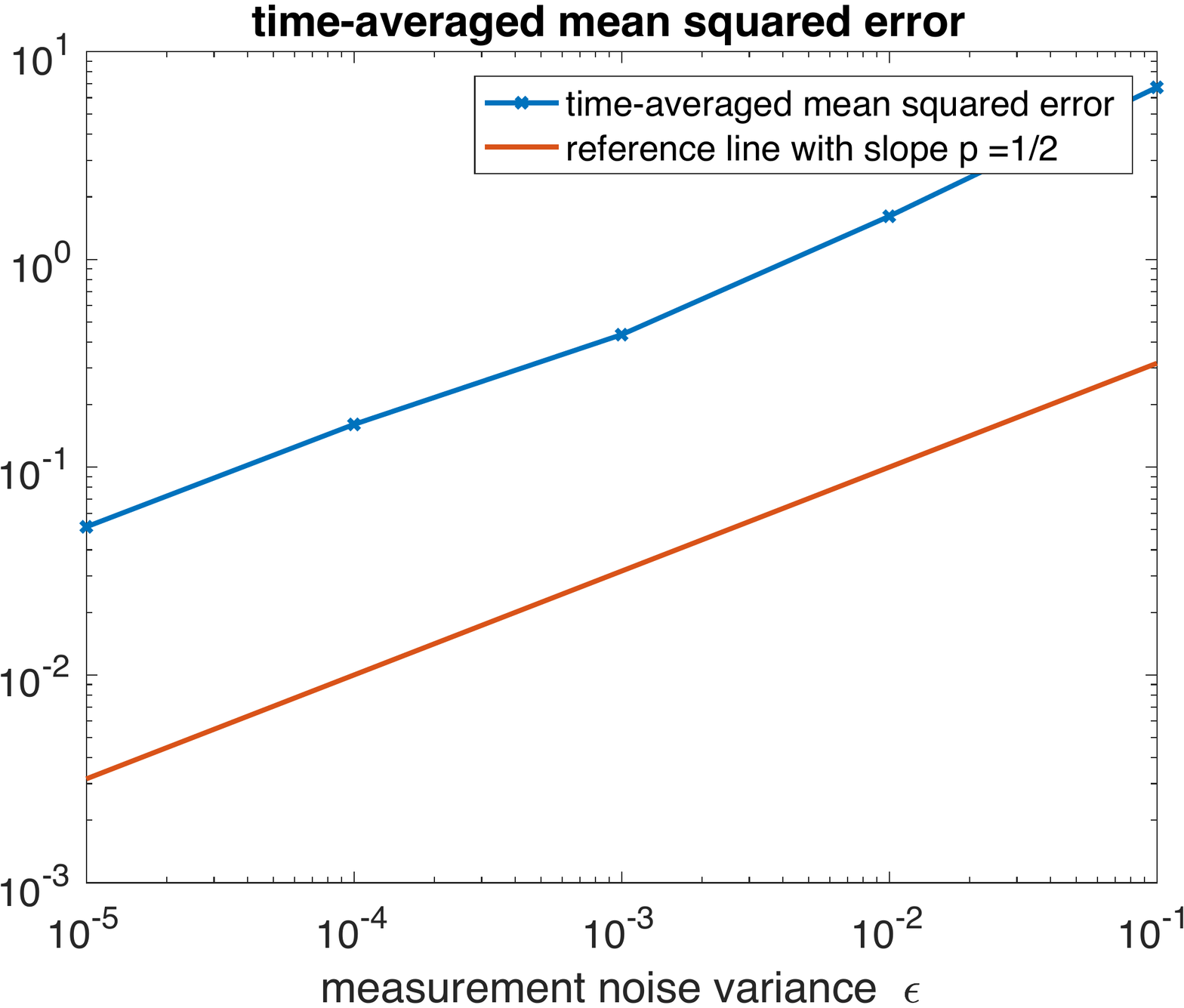} $\quad$ 
\includegraphics[trim={0 6cm 0 6cm},clip,width=0.45\textwidth]{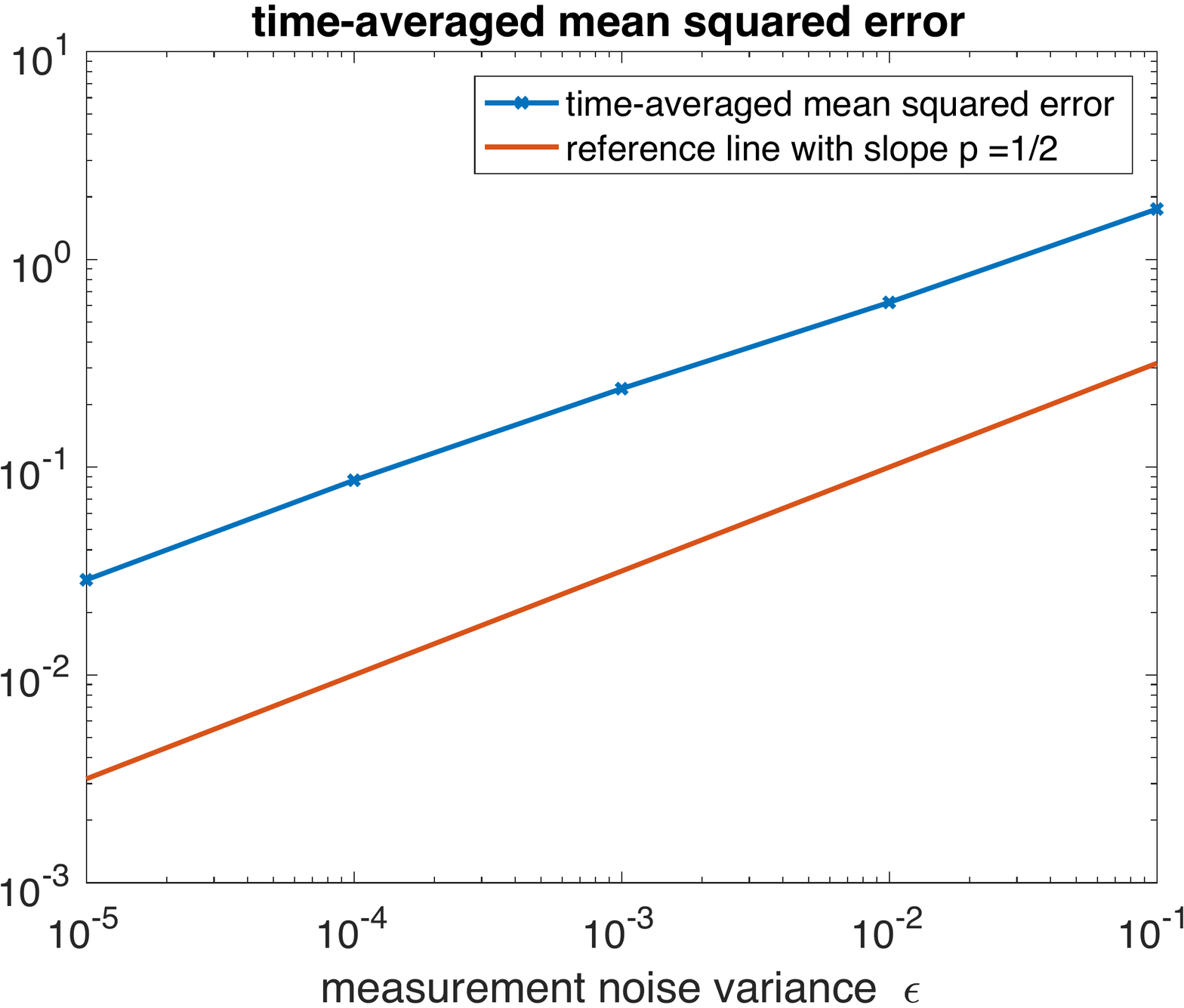}
\end{center}
\caption{Time-averaged mean squared error as a function of the measurement 
error variance $\varepsilon$ for ensemble sizes $M=2$ (left panel) and $M=3$ (right panel).} 
\label{figure3}
\end{figure}

%
\section{Conclusions}
%

In this paper, we have taken first steps towards an understanding of the long-time behavior of the ensemble Kalman-Bucy filter
and have derived limiting mean-field equations. Natural extensions include partially observed processes and 
configurations which lead to singular empirical covariance matrices $P_t^M$.
We also plan to extend our analysis to other ensemble filter algorithms, such as the stochastically perturbed ensemble 
Kalman-Bucy filter and the ensemble transform particle filter. See, for example, \cite{sr:reichcotter15} for more details.

%
\section*{Acknowledgement}
%

This research has been partially funded by 
Deutsche Forschungsgemeinschaft (DFG) through grant 
CRC 1294 \lq\lq Data Assimilation\rq\rq, Project (A02) \lq\lq 
Long-time stability and accuracy of ensemble transform
filter algorithms\rq\rq.

\bibliographystyle{alpha}
\bibliography{survey_paper}


\section*{Appendix: Supplement to the proof of Theorem \ref{MFConvergence}}

The purpose of this Appendix is to provide two Lemmata on the control of 
$\|P_t^M - {\cal P}_t\|_F$ and on the existence of exponential moments of 
$\int_0^t \frac 1M \sum_{i=1}^M \|\hat{X}_s^i\|^2 {\rm d}s$ used in the proof of Theorem 
\ref{MFConvergence}.
\begin{lemma} 
\label{lemma4} 
\begin{equation}
\label{lemma4:eq1}
\| P_t^M - \mathcal{P}_t\|_F 
\le 2 \Sigma (t)\left( \frac 1{M-1} \sum_i \|X_t^i - \hat{X}_t^i\|^2 \right)^{\frac 12}  + R_M (t)
\end{equation}
with $\lim_{M\to\infty} R_M (t)= 0$ a.s. and in $L^1 (\mathbb{P})$. Here
\begin{equation}
\Sigma (t) := \left( \frac 1{M-1} \sum_i \|X_t^i - \bar{x}_t^M\|^2 \right)^{\frac 12} 
+ \left( \frac 1{M-1} \sum_i \|\hat{X}_t^i - \bar{x}_t\|^2 \right)^{\frac 12}\,.
\end{equation}
Similarly, 
\begin{equation}
\label{lemma4:eq2}
\| Q_t^M - \mathcal{Q}_t\|_F 
\le 2(1 + \|h\|_{\rm Lip})\Sigma (t)  \left( \frac 1{M-1} \sum_i \|X_t^i - \hat{X}_t^i\|^2 \right)^{\frac 12}  + S_M (t)
\end{equation}
with $\lim_{M\to\infty} S_M (t) = 0$ a.s.~and in $L^1 (\mathbb{P})$. 
\end{lemma} 

\begin{remark} Note that the factor $\Sigma (t)$ 
is locally bounded in $t$ due to Lemma \ref{lemma2} and
an appropriate generalization of Lemma \ref{UnifMBds}.  
\end{remark}

\begin{proof} (of Lemma \ref{lemma4}) First note that we can decompose 
\begin{equation}
\begin{aligned} 
P_t^M - \mathcal{P}_t 
& = \frac 1{M-1} \sum_{i=1}^M \left( X_t^i - \bar{x}_t^M\right) \left(X_t^i - \bar{x}_t^M\right)^{\rm T}
- \mathbb{E} \left[ \left( \hat{X}_t - \bar{x}_t\right)\left(\hat{X}_t -\bar{x}_t\right)^{\rm T} \right] \\
& = \frac 1{M-1} \sum_{i=1}^M \left( X_t^i - \bar{x}_t^M - \left( \hat{X}_t^i - \bar{x}_t\right)\right) \left(X_t^i - \bar{x}_t^M\right)^{\rm T} \\ 
& \qquad + \frac 1{M-1} \sum_{i=1}^M \left( \hat{X}_t^i - \bar{x}_t\right)\left( X_t^i - \bar{x}_t^M - \left( \hat{X}_t^i - \bar{x}_t\right)\right)^{\rm T} \\
& \qquad + \frac 1{M-1} \sum_{i=1}^M \left( \hat{X}_t^i - \bar{x}_t\right)\left( \hat{X}_t^i - \bar{x}_t\right)^T - \mathbb{E} \left[ \left( \hat{X}_t - \bar{x}_t\right)\left(\hat{X}_t -\bar{x}_t\right)^{\rm T} \right] \\
& = I + II + III\, .
\end{aligned} 
\end{equation}
In particular, $\|P_t^M - \mathcal{P}_t\|_{\rm F} \le   \|I\|_{\rm F} + \|II\|_{\rm F} + \|III\|_{\rm F}$. 
Term $I$ can be estimated from above by 
\begin{equation}
\begin{aligned} 
\|I\|_{\rm F} & \le \left( \frac 1{M-1} \sum_{i=1}^M \|X_t^i - \bar{x}_t^M - \left( \hat{X}_t^i - \bar{x}_t\right)\|^2 \right)^{1/2}\left(\frac 1{M-1} \sum_{i=1}^M \|X_t^i - \bar{x}_t^M\|^2 \right)^{1/2} \\
& \le
\left( \left( \frac 1{M-1} \sum_{i=1}^M \|X_t^i - \hat{X}_t^i\|^2\right)^{1/2} 
+ \sqrt{\frac{M}{M-1}} \|x_t^M - \bar{x}_t\| \right) 
\left(\frac 1{M-1} \sum_{i=1}^M \|X_t^i - \bar{x}_t^M\|^2 \right)^{1/2} \\ 
& \le \left( 2\left(\frac 1{M-1} \sum_{i=1}^M \|X_t^i - \hat{X}_t^i\|^2\right)^{1/2} 
   + \sqrt{\frac{M}{M-1}} \left\|\frac 1M \sum_{i=1}^M \hat{X}_t^i - \mathbb{E}\left[\hat{X}_t^i\right]\right\| \right) \\ 
   & \qquad\qquad 
   \left(\frac 1{M-1} \sum_{i=1}^M \|X_t^i - \bar{x}_t^M\|^2 \right)^{1/2} \, . 
\end{aligned} 
\end{equation}
Similarly, 
\begin{equation}
\begin{aligned} 
\|II\|_{\rm F} &\le \left( 2\left(\frac 1{M-1} \sum_{i=1}^M \|X_t^i - \hat{X}_t^i\|^2\right)^{1/2} 
   + \sqrt{\frac{M}{M-1}} \left\|\frac 1M \sum_{i=1}^M \hat{X}_t^i - \mathbb{E} \left[\hat{X}_t^i\right]\right\| \right) \\
   & \qquad\qquad 
   \left(\frac 1{M-1} \sum_{i=1}^M \|\hat{X}_t^i - \bar{x}_t\|^2 \right)^{1/2} \, . 
\end{aligned} 
\end{equation}
Finally, 
\begin{equation}
\begin{aligned} 
\|III\|_{\rm F} & = \left\| \frac 1M \sum_{i=1}^M \left( \hat{X}_t^i(\hat{X}_t^{i})^{\rm T} - \mathbb{E} \left[ \hat{X}_t^i(\hat{X}_t^{i})^{\rm T} \right]\right) + \frac 1{M(M-1)} \left( \hat{X}_t^i - \bar{x}_t\right)\left(\hat{X}_t^i - \bar{x}_t\right)^{\rm T}\right\|_{\rm F} \\
& \le \left\| \frac 1M \sum_{i=1}^M \left( \hat{X}_t^i (\hat{X}_t^{i})^{\rm T} - \mathbb{E} \left[ \hat{X}_t^i(\hat{X}_t^{i})^{\rm T} \right]\right)\right\|_{\rm F}
+ \frac 1M \left\|\frac 1{M-1} \sum_{i=1}^M 
\left( \hat{X}_t^i - \bar{x}_t\right)\left(\hat{X}_t^i - \bar{x}_t\right)^{\rm T}\right\|_{\rm F} 
\end{aligned} 
\end{equation}
Adding up all terms we arrive at the estimate 
\begin{equation}
\begin{aligned} 
\|P_t^M - \mathcal{P}_t\|_{\rm F} & \le 2\Sigma (t) \left(\frac 1{M-1} \sum_{i=1}^M \|X_t^i - \hat{X}_t^i\|^2\right)^{1/2}  + R_M (t)
\end{aligned} 
\end{equation}
with the remainder 
\begin{equation}
\begin{aligned}
R_M (t) & = \Sigma (t) \sqrt{\frac{M}{M-1}} \left\|\frac 1M \sum_{i=1}^M \hat{X}_t^i - \mathbb{E} \left[\hat{X}_t^i\right]\right\| + \|III\|_{\rm F} \, . 
\end{aligned}
\end{equation}
The strong law of large numbers now implies that $\lim_{M\to\infty} R_M (t) = 0$ in a.s.~and in $L^1 (\mathbb{P})$. 
The proof of the second estimate is done similarly. 
\end{proof}

\begin{lemma} 
\label{lemma3a} 
Let $\hat{X}_t^i$, $1\le i\le M$, $M\ge 2$, be the solution of \eqref{EKB1a} with initial conditions iid ($\pi_0$) and suppose that $\pi_0$ has bounded support contained in a ball with radius $K$. Then for all $T>0$ 
there exist $\delta_0 > 0$ and $\kappa_0 >0$ depending on $T$, but independent of $M$, such that 
\begin{equation}
\mathbb{E}\left[ e^{\delta_0\int_0^t \frac 1M\sum_{i=1}^M \|\hat{X}_s^i\|^2 {\rm d}s } 
\right]
\le e^{2\kappa_0 \left( \frac {K^2}M  + \left\| h\left( X^{\rm ref}_{0:T} \right)\right\|^2_\infty \right)} < +\infty\quad\forall\, t\le T\, . 
\end{equation}
Here, the expectation is taken also w.r.t.~the distribution of $\{ Y_s\}$. 
\end{lemma} 

\begin{proof} 
First note that It\^o's formula and \eqref{EKB1a} imply that 
\begin{equation}
\begin{aligned} 
{\rm d} \left( \frac 1M \sum_{i=1}^M \|\hat{X}_t^i\|^2\right) 
& = \frac 2M \sum_{i=1}^M \langle f\left( \hat{X}_t^i\right) , \hat{X}_t^i\rangle {\rm d}t  
+  \frac 2M \sum_{i=1}^M \langle D\mathcal{P}_t^{-1}\left( \hat{X}_t^i - \bar{x}_t, \right) , \hat{X}_t^i\rangle {\rm d}t \\ 
& \qquad - \frac 1M \sum_{i=1}^M \langle \mathcal{Q}_t R^{-1} h\left( \hat{X}_t^i\right) , \hat{X}_t^i\rangle {\rm d}t  
  - \frac 1M \sum_{i=1}^M \langle \mathcal{Q}_t R^{-1} \bar{h}_t , \hat{X}_t^i\rangle {\rm d}t \\
  & \qquad 
 +  \frac 2M \sum_{i=1}^M \langle \hat{X}_t^i, \mathcal{Q}_t R^{-1} {\rm d}Y_t \rangle + \frac 1M \mbox{tr} \left( \mathcal{Q}_t R^{-1}\mathcal{Q}_t\right){\rm d}t \, . 
\end{aligned}
\end{equation}
Using Lipschitz continuity of $f$ and $h$ and the previous two Lemmata \ref{lemma2} and \ref{lemma3}, the right hand side can be estimated from above for $t\le T$ by 
\begin{equation}
C(T) \left( 1 + \frac 1M \sum_{i=1}^M \|\hat{X}_t^i\|^2 \right) 
+ \frac 2M \sum_{i=1}^M \langle \hat{X}_t^i, \mathcal{Q}_t R^{-1} {\rm d}Y_t \rangle 
\end{equation}
for some uniform constant $C(T)$. Since 
${\rm d}Y_t = h\left( X^{\rm ref}_t \right) {\rm d}t + R^{-1/2}{\rm d}B_t$ 
we can further estimate from above for $t\le T$  
\begin{equation}
C(T) \left( 1 + \left\| h\left( X^{\rm ref}_t \right)\right\|^2 + \frac 1M \sum_{i=1}^M \|\hat{X}_t^i\|^2 \right) 
+ \frac 2M \sum_{i=1}^M \langle \hat{X}_t^i, \mathcal{Q}_t R^{-1/2}{\rm d}B_t \rangle 
\end{equation}
for some possibly different constant $C(T)$. It\^o's product rule now implies for $\alpha := 1 + C(T)$ and $t\le T$ 
\begin{equation}
\begin{aligned} 
{\rm d}\left(  e^{-\alpha t}\frac 1M \sum_{i=1}^M \|\hat{X}_t^i\|^2 \right) 
& \le e^{- \alpha t} C(T) \left( 1 + 
\left\| h\left( X^{\rm ref}_t \right)\right\|^2 \right)\, {\rm d} t 
- e^{-\alpha t} \left( \frac 1M \sum_{i=1}^M \|\hat{X}_t^i\|^2 \right){\rm d}t \\
& \qquad 
+ e^{-\alpha t}\frac 2M \sum_{i=1}^M \langle \hat{X}_t^i, \mathcal{Q}_t R^{-1/2} {\rm d}B_t \rangle \,,
\end{aligned} 
\end{equation}
which implies that 
\begin{equation}
\label{lemma3a_est1}
\begin{aligned} 
\int_0^t e^{-\alpha s}\frac 1M \sum_{i=1}^M \|\hat{X}_s^i\|^2 {\rm d}s 
& \le \frac 1M \sum_{i=1}^M \|\hat{X}_0^i\|^2  
+ C(T) \left(1 + \left\| h\left( X^{\rm ref}_{0:T} \right)\right\|^2_\infty \right) \\
& \qquad + \int_0^t e^{-\alpha s} \frac 2M \sum_{i=1}^M \langle \hat{X}_s^i, \mathcal{Q}_t R^{-1/2}{\rm d}B_s \rangle \, . 
\end{aligned} 
\end{equation} 

To simplify notations in the following let  
\begin{equation}
M_t := \int_0^t e^{-\alpha s} \frac 2M \sum_{i=1}^M \langle \hat{X}_s^i, \mathcal{Q}_t R^{-1/2}{\rm d}B_s \rangle 
\end{equation}
and observe that the quadratic variation $\langle M \rangle_t$  can be estimated from above by  
\begin{equation}
\begin{aligned}
\langle M \rangle_t  &= \frac 4{M^2} \sum_{i=1}^M \int_0^t e^{-2\alpha s} \|R^{-1/2} \mathcal{Q}^T_s \hat{X}_s^i\|^2{\rm d}s \\
&\le \frac{4\|R^{-1/2}\|^2_{\rm F}\|h\|_{\rm Lip}^2 C(T)^2}M \int_0^t e^{-\alpha s} \frac 1M \sum_{i=1}^M \|\hat{X}_s^i\|^2{\rm d}s  \,,
\end{aligned}
\end{equation}
using 
\begin{equation}
\| \mathcal{Q}_s\|^2_{\rm F} 
\le \|h\|_{\rm Lip}^2 \mathbb{E}\left[  \|\hat{X}_s - \bar{x}_s\|^2\right] 
\le \|h\|_{\rm Lip}^2 C(T)^2 
\end{equation} 
and Lemma \ref{lemma2}. The assumption on the initial condition now implies for $\delta > 0$  
\begin{equation} 
\begin{aligned} 
\mathbb{E}\left[ e^{\delta \int_0^t e^{-\alpha s}\frac 1M \sum_{i=1}^M \|\hat{X}_s^i\|^2 \, ds }\right]
& \le e^{\delta \left( \frac {K^2}M 
+  C(T)\left(1 +  \left\| h\left( X^{\rm ref}_{0:T} \right)\right\|^2_\infty \right) \right)} \mathbb{E}\left[ e^{ \delta M_t } \right] \\ 
& \le e^{\delta \left( \frac {K^2}M
+  C(T)\left( 1 + \left\| h\left( X^{\rm ref}_{0:T} \right)\right\|^2_\infty \right) \right)} \mathbb{E}\left[ e^{2\delta^2 \langle M\rangle_t} \right]^{1/2} \\  
& \le e^{\delta \left( \frac{K^2}M  
+  C(T)\left( 1 +  \left\| h\left( X^{\rm ref}_{0:T} \right)\right\|^2_\infty \right) \right)} \\
 & \qquad  \mathbb{E}\left[ e^{2\delta \frac{4\|R^{-1/2}\|^2_{\rm F} \|h\|_{\rm Lip}^2 C(T)^2}M \delta \int_0^t e^{-\alpha s} \frac 1M \sum_{i=1}^M \|\hat{X}_s^i\|^2{\rm d}s}\right]^{1/2} \, ,  
\end{aligned} 
\end{equation}
thereby using the inequality 
\begin{equation}
\begin{aligned} 
\mathbb{E}\left[ e^{\delta M_t} \right] 
& = \mathbb{E}\left[ e^{\frac 12 \left( 2\delta M_t - 2\delta^2 \langle M\rangle_t\right)} e^{\frac 12 \left( 2\delta^2 \langle M\rangle_t\right)} \right]  
\le \mathbb{E}\left[ e^{ 2\delta M_t - 2\delta^2 \langle M\rangle_t}\right]^{1/2} \mathbb{E}\left[ e^{2\delta^2 \langle M\rangle_t}\right]^{1/2} \\
& = \mathbb{E}\left[ e^{2\delta^2 \langle M\rangle_t}\right]^{1/2} 
\, .  
\end{aligned}  
\end{equation}
Hence for $\delta_0 > 0$ with 
\begin{equation}
\delta_0 \frac{8\|R^{-1/2}\|_{\rm F}^2 \|h\|_{\rm Lip}^2 C(T)^2}M  < 1 
\end{equation}
it follows that 
\begin{equation}
\begin{aligned} 
\mathbb{E}\left[ e^{\delta_0 \int_0^t e^{-\alpha s}\frac 1M \sum_{i=1}^M \|\hat{X}_s^i\|^2 {\rm d}s }\right]
& \le e^{2\delta_0 \left( \frac{K^2}M   
+ C(T)\left(1 + \left\| h\left( X^{\rm ref}_{0:T} \right)\right\|^2_\infty \right) \right)} 
< e^{2\kappa_0 \left( \frac{K^2}M   
+   \left\| h\left( X^{\rm ref}_{0:T} \right)\right\|^2_\infty \right)} 
< +\infty\
\end{aligned} 
\end{equation}
for a suitable $\kappa_0 >0$.
\end{proof} 

\end{document}